\documentclass[10pt]{amsart}   	
									
\usepackage{amsmath}
\usepackage{amssymb, amsfonts}
\usepackage{amsthm}
\usepackage[hidelinks]{hyperref}
\usepackage{mathtools}
\usepackage{tikz}
\usetikzlibrary{cd}
\usetikzlibrary{decorations.markings}
\usepackage{enumitem}
\usepackage{makecell}
\setcellgapes{4pt}

\usepackage[margin=1in]{geometry}

\newtheorem{thm}{Theorem}[section]
\newtheorem*{thm*}{Theorem}
\newtheorem{lem}[thm]{Lemma}
\newtheorem{cor}[thm]{Corollary}
\newtheorem{prop}[thm]{Proposition}
\newtheorem{conj}[thm]{Conjecture}

\newtheorem{innercustomthm}{Theorem}
\newenvironment{customthm}[1]
{\renewcommand\theinnercustomthm{#1}\innercustomthm}{\endinnercustomthm}

\theoremstyle{definition}
\newtheorem{define}[thm]{Definition}
\newtheorem{defthm}[thm]{Definition-Theorem}
\newtheorem*{define*}{Definition}
\newtheorem{ex}[thm]{Example}

\newtheorem*{nota*}{Notation}
\newtheorem{rem}[thm]{Remark}
\newtheorem*{rem*}{Remark}

\newtheorem{ques}{Question}

\newcommand{\Z}{\mathbb{Z}}

\newcommand{\inv}{^{-1}}

\newcommand{\B}{\mathcal{B}}

\newcommand{\X}{\mathcal{X}}
\newcommand{\Y}{\mathcal{Y}}
\newcommand{\T}{\mathcal{T}}
\renewcommand{\S}{\mathcal{S}}
\newcommand{\Cat}{\mathcal{C}}

\renewcommand{\H}{\mathcal{H}_{0\textnormal{-PS}}}
\renewcommand{\P}{\mathcal{P}}
\newcommand{\M}{\mathcal{M}}

\newcommand{\Hom}{\mathrm{Hom}}
\newcommand{\Ext}{\mathrm{Ext}}
\newcommand{\End}{\mathrm{End}}

\newcommand{\image}{\mathrm{Im}}
\renewcommand{\dim}{\mathrm{dim}}
\newcommand{\supp}{\mathrm{supp}}
\renewcommand{\ker}{\mathrm{ker}}
\newcommand{\coker}{\mathrm{coker}}

\newcommand{\mods}{\mathsf{mod}}

\newcommand{\Db}{\mathcal{D}^b}
\newcommand{\W}{\mathfrak{W}}

\newcommand{\wide}{\mathsf{wide}}
\newcommand{\tors}{\mathsf{tors}}

\newcommand{\sbrick}{\mathsf{sbrick}}
\newcommand{\brick}{\mathsf{brick}}
\newcommand{\str}{\mathsf{s}\tau\textnormal{-}\mathsf{rigid}}

\newcommand{\smc}{2\textnormal{-}\mathsf{smc}}
\newcommand{\add}{\mathsf{add}}
\newcommand{\Filt}{\mathsf{Filt}}

\newcommand{\Fac}{\mathsf{Fac}}
\newcommand{\Sub}{\mathsf{Sub}}
\newcommand{\ind}{\mathsf{ind}}

\newcommand{\thick}{\mathsf{thick}}
\newcommand{\Hasse}{\mathsf{Hasse}}

\newcommand{\Arc}{\mathsf{Arc}}

\newcommand{\cone}{\mathsf{cone}}
\newcommand{\cocone}{\mathsf{cocone}}

\newcommand{\lperp}[1]{\prescript{\perp}{}{#1}}

\newcommand{\xrightarrowtwohead}[2][]{%
  \xrightarrow[#1]{#2}\mathrel{\mkern-14mu}\rightarrow
}

\usepackage{soul}

\title{Pairwise Compatibility for 2-Simple Minded Collections}
\author{Eric J. Hanson}
\address{Brandeis University, Department of Mathematics, 415 South Street, Waltham MA 02453, USA}
\email{ehanson4@brandeis.edu}

\author{Kiyoshi Igusa}
\address{Brandeis University, Department of Mathematics, 415 South Street, Waltham MA 02453, USA}
\email{igusa@brandeis.edu}



\date{October 14, 2020}

\thanks{\copyright \ 2020 Elsevier. This work is licensed under CC-BY-NC-ND 2.0. See https://creativecommons.org/licenses/by-nc-nd/2.0/ for a copy of this license.}

\begin{document}
\noindent J. Pure Appl. Algebra 225 (2021), no. 6. \href{https://doi.org/10.1016/j.jpaa.2020.106598}{DOI:10.1016/j.jpaa.2020.106598}.

\bigskip

\maketitle

\begin{abstract}
	In $\tau$-tilting theory, it is often difficult to determine when a set of bricks forms a 2-simple minded collection. The aim of this paper is to determine when a set of bricks is contained in a 2-simple minded collection for a $\tau$-tilting finite algebra. We begin by extending the definition of mutation from 2-simple minded collections to more general sets of bricks (which we call semibrick pairs). This gives us an algorithm to check if a semibrick pair is contained in a 2-simple minded collection. We then use this algorithm to show that the 2-simple minded collections of a $\tau$-tilting finite gentle algebra (whose quiver contains no loops or 2-cycles) are given by pairwise compatibility conditions if and only if every vertex in the corresponding quiver has degree at most 2. As an application, we show that the classifying space of the $\tau$-cluster morphism category of a $\tau$-tilting finite gentle algebra (whose quiver contains no loops or 2-cycles) is an Eilenberg-MacLane space if every vertex in the corresponding quiver has degree at most 2.
\end{abstract}

\tableofcontents

\section*{Introduction} 
This paper furthers the study of the connection between \emph{$\tau$-tilting theory} and \emph{semi-invariant pictures} described in \cite{hanson_tau}. In that paper, the authors study a finitely presented group, called the \emph{picture group}, associated to an arbitrary $\tau$-tilting finite algebra. Picture groups (and picture spaces) were first defined in a special case by Loday in \cite{loday_homotopical}, then by the second author, Todorov, and Weyman in the general hereditary case in \cite{igusa_picture}. In \cite{hanson_tau}, this is extended to the non-hereditary case. The picture group of an algebra can be realized as the fundamental group of a topological space, constructed as the classifying space of the \emph{$\tau$-cluster morphism category} of the algebra, as defined by Buan and Marsh in \cite{buan_category} to generalize a construction of the second author and Todorov in \cite{igusa_signed}.

Crucial to relating the cohomology of the picture group of an algebra to the cohomology of the associated topological space is an understanding of the algebra's \emph{2-simple minded collections}, as defined in \cite{koenig_silting,brustle_ordered}. These collections generalize the idea of a complete collection of non-isomorphic simple modules, and are known to be in bijection with many other objects in representation theory, for example 2-term silting complexes, support $\tau$-tilting objects, and functorially finite torsion classes (see \cite{brustle_ordered,asai_semibricks} for a more detailed list). In \cite{hanson_tau}, the authors show that the topological space and picture group have isomorphic cohomology when the 2-simple minded collections of the algebra are given by pairwise compatibility conditions, generalizing results of \cite{igusa_signed} and \cite{igusa_category}.

One of the main results of this paper is to show that for most $\tau$-tilting finite gentle algebras, the 2-simple minded collections cannot be defined using pairwise compatibility conditions, disproving a conjecture from \cite{hanson_tau}. This is in stark contrast to many of the other associated structures in representation theory. For example, both support $\tau$-tilting objects (see \cite[Thm. 2.12]{adachi_tilting}) and 2-term silting objects (see \cite[Prop. 2.16]{aihara_tilting}) are given by pairwise compatibility conditions. In light of this, our results indicate that 2-simple minded collections are in many ways much more subtle than some of these other structures.

Gentle algebras form a natural class to study due to the long known combinatorial description of their indecomposable modules and the morphisms between them in terms of strings (see \cite{butler_auslander,crawley_maps,schroer_modules}). More recently, a basis has been given for extensions between indecomposable modules in \cite{canakci_extensions,brustle_combinatorics} and the $\tau$-tilting theory of gentle algebras has been described in terms of non-kissing complexes (see \cite{palu_kissing,brustle_combinatorics}). This has influenced further study of the relationship between gentle algebras and various combinatorial objects, such as ribbon graphs (see \cite{schroll_trivial,opper_geometric}), marked punctured surfaces (see \cite{baur_geometric,lekili_derived,palu_kissing,amiot_complete}), and biclosed sets (see \cite{garver_categorification,garver_oriented}). One of the central results of \cite{garver_oriented} relates the 2-simple minded collections of certain gentle algebras to the data of a noncrossing tree partition and its Kreweras complement. We remark that our work does not rely on this interpretation.

\subsection*{Notation and Terminology}
Let $\Lambda$ be a finite dimensional, basic algebra over an arbitrary field $K$. When we write $\Lambda = KQ/I$, we assume that $Q$ is a quiver and that $I$ is an admissible ideal unless otherwise stated. Our convention is to multiply paths left to right. We denote by $\mods\Lambda$ the category of finitely generated (right) $\Lambda$-modules. Throughout this paper, all subcategories are assumed to be full and closed under isomorphisms. For $M \in \mods\Lambda$, we denote by $\add M$ (resp. $\Fac M, \Sub M)$ the subcategory of direct summands (resp. factors, submodules) of finite direct sums of $M$. Moreover, $\Filt M$ refers to the subcategory of objects admitting a (finite) filtration by the direct summands of $M$. Given a subcategory $\M \subset \mods\Lambda$, we define $\add\M, \Fac\M, \Sub\M$, and $\Filt\M$ analogously.

We denote by $\Db(\mods\Lambda)$ the bounded derived category of $\mods\Lambda$. The symbol $(-)[1]$ will denote the shift functor in all triangulated categories. We identify $\mods\Lambda$ with the subcategory of $\Db(\mods\Lambda)$ consisting of stalk complexes centered at zero.

For an object $M$ in a category $\Cat$, we define the \emph{left-perpendicular category} of $M$ as $\lperp{M}:=\{N \in \Cat|\Hom(N,M) = 0\}$. We define the \emph{right-perpendicular category}, $M^\perp$, dually. For a subcategory $\M \subset \Cat$, we define $\lperp{\M}$ and $\M^\perp$ analogously. If $M \in N^\perp \cap \lperp{N}$ we say the objects $M$ and $N$ are \emph{Hom orthogonal}. Likewise if the category $\Cat$ is triangulated, $N[1] \in M^\perp$, and $M[1] \in N^\perp$, we say $M$ and $N$ are \emph{Ext orthogonal}. We say $M$ and $N$ are \emph{Hom-Ext orthogonal} if they are both Hom and Ext orthogonal. We denote by $\ind(\Cat)$ the category of indecomposable objects of $\Cat$.

\subsection*{Organization and Main Results}
The contents of this paper are as follows. In Section \ref{sec:background}, we recall the definitions and preliminary results we will use regarding semibricks, 2-simple minded collections, and gentle algebras. We also give the definition of a semibrick pair, which drops from the definition of a 2-simple minded collection the assumption that it generates the bounded derived category.

In Section \ref{sec:mutationcompatibility}, we define a notion of mutation for certain well-behaved semibrick pairs that agrees with that for 2-simple minded collections. This allows us to prove our first main theorem.

\begin{customthm}{A}[Theorem \ref{thm:mutationcomplete}]
	Let $\Lambda$ be $\tau$-tilting finite. Then a semibrick pair of $\Lambda$ is a subset of a 2-simple minded collection if and only if it is \emph{mutation compatible} (see Definitions \ref{def:mutation}, \ref{def:mutcompat}).
\end{customthm}

We conclude Section \ref{sec:mutationcompatibility} by showing that the 2-simple minded collections of representation finite hereditary algebras can be defined using pairwise compatibility conditions, giving a new proof of a result of \cite{igusa_signed}.

In Section \ref{sec:cyclics}, we discuss a class of algebras we refer to as \emph{Nakayama-like algebras} (Definition \ref{def:nakayamalike}). In particular, we use our results on mutation compatibility to prove our second main theorem.

\begin{customthm}{B}[Theorem \ref{thm:cyclichaveproperty}]
Let $\Lambda$ be a \emph{Nakayama-like algebra}. Then the 2-simple minded collections of $\Lambda$ can be defined using pairwise compatibility conditions.
\end{customthm}

This implies a known result, namely this holds from \cite{hanson_tau} in the case that $\Lambda$ is Nakayama and from \cite{igusa_signed}, \cite{igusa_category} in the case that $\Lambda \cong KA_n$.

In Section \ref{sec:gentleclassification}, we use our results on mutation compatibility to prove the central theorem of this paper, disproving a conjecture from \cite{hanson_tau}.

\begin{customthm}{C}[Theorem \ref{thm:gentleclassification}]
	Let $\Lambda = KQ/I$ be a $\tau$-tilting finite gentle algebra such that $Q$ contains no loops or 2-cycles. Then the 2-simple minded collections of $\Lambda$ can be defined using pairwise compatibility conditions if and only if every vertex of $Q$ has degree at most 2.
\end{customthm}

Amongst the simplest algebras for which the 2-simple minded collections cannot be defined using pairwise compatibility conditions are cluster tilted algebras of type $A_n$ for $n \geq 4$. This is in contrast to the cyclic cluster tilted algebras of type $D_n$, which were shown to have this property in \cite{hanson_tau}.

In Section \ref{sec:application}, we discuss picture groups and construct faithful group functors for Nakayama-like and gentle algebras. We do so by constructing a group homomorphism from the picture group to the group of units of the \emph{power series 0-Hall algebra} of $\Lambda$ (Definition \ref{def:hallalg}), a variant of the Hall algebra constructed by evaluating the Hall polynomials at $q=0$. This, together with the results on pairwise compatibility, allows us to prove our final main theorem.

\begin{customthm}{D}[Corollary \ref{cor:kpi1}]\
	\begin{enumerate}[label=\upshape(\alph*)]
		\item Let $\Lambda$ be a Nakayama-like algebra. Then the classifying space of the \emph{$\tau$-cluster morphism category} of $\Lambda$ is a $K(\pi,1)$ for the picture group of $\Lambda$.
		\item Let $\Lambda = KQ/I$ be a $\tau$-tilting finite gentle algebra such that $Q$ contains no loops or 2-cycles. Then the classifying space of the \emph{$\tau$-cluster morphism category} of $\Lambda$ is a $K(\pi,1)$ for the picture group of $\Lambda$ if every vertex of $Q$ has degree at most 2.
	\end{enumerate}
\end{customthm}

\section{Background}\label{sec:background}

Recall that a subcategory $\T \subset \mods\Lambda$ is called a \emph{torsion class} if it is closed under extensions and factors. Likewise, a subcategory $W \subset \mods\Lambda$ is called a \emph{wide subcategory} if it is closed under extensions, kernels, and cokernels. We denote by $\tors\Lambda$ (resp. $\wide\Lambda$) the poset of torsion classes (resp. wide subcategories) ordered by inclusion. We assume throughout this paper that $\tors\Lambda$ is a finite set, or equivalently (see \cite[Thm 3.8]{demonet_tilting}) that $\Lambda$ is $\tau$-tilting finite.

\subsection{Semibricks and 2-Simple Minded Collections}

Recall that a (necessarily indecomposable) object $S \in \mods\Lambda$ (or more generally $\Db(\mods\Lambda)$) is called a \emph{brick} if $\End(S)$ is a division algebra. A set of bricks $\S$ is called a \emph{semibrick} if it consists of pairwise Hom-orthogonal bricks. Depending on context, the term semibrick can refer to either the set $\S$ or the object $\displaystyle \bigsqcup_{S\in\S}S \in \Db(\mods\Lambda)$. We denote by $\brick\Lambda$ (resp. $\sbrick\Lambda$) the set of isoclasses of bricks (resp. semibricks) in $\mods\Lambda$. Central to this paper are the closely related \emph{2-simple minded collections}, defined as follows by \cite[Rmk. 4.11]{brustle_ordered}.

\begin{define}
	Let $\X = \S_p \sqcup \S_n[1]$ with $\S_p, \S_n \in \sbrick\Lambda$. Then $\X$ is called a \emph{2-simple minded collection} if
	\begin{enumerate}[label=\upshape(\alph*)]
		\item For all $S \neq T \in \X$, we have $\Hom(S,T[m]) = 0$ for $m\in \{-1,0\}$. Equivalently, $\Hom(\S_p,\S_n[m])$ for $m \in \{0,1\}$.
		\item $\thick(\X) = \Db(\mods\Lambda)$, where $\thick(\X)$ is the smallest triangulated subcategory of $\Db(\mods\Lambda)$ containing $\X$ which is closed under direct summands.
	\end{enumerate}
\end{define}

\begin{rem}
	2-simple minded collections are examples of the more general \emph{simple-minded collections} (see \cite{koenig_silting}). In this general context, condition (a) is replaced with the requirement that $\Hom(S,T[m]) = 0 $ for $m \leq 0$. In our context, either $S$ or $S[-1]$ is a module (and likewise for $T$), so we need only consider $m \in \{-1,0\}$.
\end{rem}

We denote by $\smc\Lambda$ the set of isoclasses of 2-simple minded collections for $\Lambda$. This leads us to the following definition.

\begin{define}\cite[Def. 1.8]{hanson_tau}\label{def:sbp}
	Let $\S_p,\S_n \in \sbrick\Lambda$. We say that $\X = \S_p\sqcup \S_n[1]$ is a \emph{semibrick pair} if $\Hom(\S_p,\S_n[m]) = 0$ for $m \in \{0,1\}$ (or equivalently, $\Hom(S,T[m]) = 0$ for all $S\neq T \in \X$ and $m \in \{-1,0\}$). In particular, a semibrick pair $\X = \S_p\sqcup \S_n[1]$ is a 2-simple minded collection if and only if $\thick(\X) = \Db(\mods\Lambda)$. We say the semibrick pair $\S_p \sqcup \S_n[1]$ is \emph{completable} if it is a subset of a 2-simple minded collection.
\end{define}

It is in general difficult to determine when a semibrick pair is completable; we do, however, have the following, deduced from \cite[Thm. 2.3]{asai_semibricks}. We remark that as stated, this result depends on the fact that $\Lambda$ is $\tau$-tilting finite.

\begin{thm}\label{thm:semibrickcompletable}
	Let $\X = \S_p \sqcup \S_n[1]$ be a semibrick pair. If $\S_p = \emptyset$ or $\S_n = \emptyset$, then $\X$ is completable.
\end{thm}

We conclude this section with the following result describing the relationship between semibricks and torsion classes. As before, our version of this statement requires that $\Lambda$ be $\tau$-tilting finite.

\begin{prop}\cite[Prop. 3.2.5, Cor. 3.2.7]{barnard_minimal}
	There is a bijection $\sbrick\Lambda \rightarrow \tors\Lambda$ given by $\S \mapsto \Filt(\Fac\S)$.
\end{prop}

\subsection{String and Gentle Algebras}
Two of the central classes of algebras studied in this paper are string algebras and the subclass of gentle algebras. We begin with their definitions.

\begin{define}\
	\begin{enumerate}[label=\upshape(\alph*)]
	\item A finite dimensional algebra $\Lambda = KQ/I$ is called a \emph{string algebra} if
		\begin{enumerate}[label=\upshape(\roman*)]
			\item The ideal $I$ is generated by monomials.
			\item Every vertex of $Q$ is the source of at most two arrows and the target of at most two arrows.
			\item For every arrow $\alpha \in Q_1$, there is at most one $\beta \in Q_1$ such that $\alpha\beta \notin I$ and at most one $\gamma \in Q_1$ such that $\gamma\alpha \notin I$.
		\end{enumerate}
	\item A string algebra is called a \emph{gentle algebra} if
		\begin{enumerate}[label=\upshape(\roman*)]
			\item The ideal $I$ is generated by paths of length two.
			\item For every arrow $\alpha \in Q_1$, there is at most one $\beta \in Q_1$ such that $0 \neq \alpha\beta \in I$ and at most one $\gamma \in Q_1$ such that $0 \neq \gamma\alpha \in I$.
		\end{enumerate}
	\end{enumerate}
\end{define}

In this paper, we will be interested only in gentle algebras whose quivers contain no loops or oriented 2-cycles. As we are assuming our algebras are all $\tau$-tilting finite, we have already excluded all algebras whose quivers contain multiple arrows with the same source and target. Under these restrictions, we see that an algebra $\Lambda = KQ/I$ is gentle if and only if for all $x \in Q_0$, the local picture at $x$ containing all arrows incident to $x$ is a subquiver of the following, where the dotted lines represent relations and all five vertices are distinct.

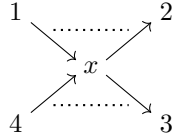
\begin{center}
\begin{tikzpicture}
	\node at (0,0.75) {1};
	\node at (0,-0.75) {4};
	\node at (1,0) {$x$};
	\node at (2,-0.75) {3};
	\node at (2,0.75) {2};
	\draw [->] (0.2,0.6)--(0.8,0.1);
	\draw [->] (0.2,-0.6)--(0.8,-0.1);
	\draw [->] (1.2,-0.1)--(1.8,-0.6);
	\draw [->] (1.2,0.1)--(1.8,0.6);
	\draw[thick,dotted] plot [smooth] coordinates{(0.5,-0.5)(1.5,-0.5)};
	\draw[thick,dotted] plot [smooth] coordinates{(0.5,0.5)(1.5,0.5)};
\end{tikzpicture}
\end{center}

The following standard definition allows for a simple description of the indecomposable modules of a string algebra.

\begin{define}\label{def:strings}
	Let $\Lambda = KQ/I$ be a string algebra. A sequence $\gamma_1^{\epsilon_1}\cdots\gamma_m^{\epsilon_m}$, where each $\gamma_i \in Q_1$ and $\epsilon_i \in \{\pm 1\}$ is called a \emph{string} if
	\begin{enumerate}[label=\upshape(\alph*)]
		\item There is no subsequence of the form $\gamma_i^{+1}\gamma_i^{-1}$ or $\gamma_i^{-1}\gamma_i^{+1}$
		\item Considering $\gamma_i^{+1}$ as the arrow $\gamma_i$ and $\gamma_i^{-1}$ as its formal inverse, we have $t(\gamma_i^{\epsilon_i}) = s(\gamma_{i+1}^{\epsilon_{i+1}})$ for $1 \leq i < m-1$.
		\item No subsequence (or its inverse) belongs to the ideal $I$.
	\end{enumerate}
	We also allow for sequences of the form $e_j$, the constant string at the vertex $j \in Q_0$.
\end{define}

We will consider string algebras which are representation finite. In this case, it is a well-known result of \cite{butler_auslander} (see also \cite{gelfand_indecomposable}) that there is a bijection between strings of $KQ/I$ (up to taking inverses) and isoclasses of $\ind(\mods\Lambda)$. This bijection allows for a nice combinatorial description of the morphisms between irreducible modules, given first more generally in \cite{crawley_maps} and later for string algebras in \cite{schroer_modules}. In lieu of summarizing this construction here, we will cite results about the construction as necessary. We refer interested readers to \cite[Sec. 2]{brustle_combinatorics}, which contains a well-written summary.

We end this section with the following result that will be crucial in what follows.

\begin{thm}\cite[Thm. 1.1]{plamondon_tilting}\label{thm:finitegentle}
	A gentle algebra is $\tau$-tilting finite if and only if it is representation finite.
\end{thm}

In particular, this means that in order for a gentle algebra to be $\tau$-tilting finite, every (not necessarily oriented) cycle of $Q$ must contain a relation.

\section{Mutation Compatibility of Semibrick Pairs}\label{sec:mutationcompatibility}

The goal of this section is to determine when a semibrick pair is completable. To do so, we define a notion of mutation for semibrick pairs. This intentionally mirrors the corresponding notion for 2-simple minded collections defined by Koenig and Yang in \cite[Sec. 7.2]{koenig_silting}. We begin with the following generalization of Lemma 7.8 in their paper. We remark that our proof depends on the fact that $\Lambda$ is $\tau$-tilting finite.

To simplify notation, for $T \in \mods\Lambda\sqcup \mods\Lambda[1]$ we denote by $|T|$ the underlying module of $T$; that is,
$$|T| = \begin{cases} T & T \in \mods\Lambda\\T[-1] & T \in \mods\Lambda[1]\end{cases}.$$

\begin{lem}\label{lem:approxconditions}
	Let $\X = \S_p \sqcup \S_n[1]$ be a semibrick pair. Then we have the following.
	\begin{enumerate}[label=\upshape(\alph*)]
		\item For $S \in \S_p$ and $S\neq T \in \X$, there exists a left minimal $(\Filt S)$-approximation $T[-1] \xrightarrow{g_{S,|T|}^+} S_{|T|}$. Moreover,
		\begin{enumerate}[label=\upshape(\roman*)]
			\item For any $X \in \Filt(S)$, the induced map $\Hom(S_{|T|},X) \rightarrow \Hom(T[-1],X)$ is a bijection.
			\item The induced map $\Hom(S_{|T|},S[1]) \rightarrow \Hom(T[-1],S[1])$ is injective.
		\end{enumerate}
		\item For $S[1] \in \S_n[1]$ and $S\neq T \in \X$, there exists a right minimal $(\Filt S)$-approximation $S_{|T|} \xrightarrow{g^-_{|T|,S}} T$. Moreover,
		\begin{enumerate}[label=\upshape(\roman*)]
			\item For any $X \in \Filt(S)$, the induced map $\Hom(X,S_{|T|}) \rightarrow \Hom(X,T)$ is a bijection.
			\item The induced map $\Hom(S,S_{|T|}[1]) \rightarrow \Hom(S,T[1])$ is injective.
		\end{enumerate}
	\end{enumerate}
\end{lem}

\begin{proof}
	We show only (a) as the proof of (b) is similar. Let $S \in \S_p$ and $S \neq T \in \X$. If $T \in \S_p$, then $S \sqcup T$ is completable by Theorem \ref{thm:semibrickcompletable}, and hence the result follows from \cite[Lem. 7.8]{koenig_silting}. Thus, we can assume $T \in \S_n[1]$, so both $S$ and $T[-1]$ are modules. Moreover, since $\Lambda$ is $\tau$-tilting finite, the wide subcategory $\Filt S$ is functorially finite by \cite[Cor. 3.11]{marks_torsion} and \cite[Thm. 3.8]{demonet_tilting}. Therefore, there exists a left minimal $(\Filt S)$-approximation $T[-1] \xrightarrow{g_{S,|T|}^+} S_T$.
	
	(i): Let $X \in \Filt(S)$ and let $f : S_{|T|} \rightarrow X$ such that $f\circ g^+_{S,|T|} = 0$. This means $\image(g_{S,|T|}^+) \subset \ker(f)$. Moreover, we see that $\ker(f) \in \Filt S$. Thus since $g^+_{S,|T|}$ is a left minimal $(\Filt S)$-approximation, it must be the case that $\ker(f) = S_{|T|}$ and hence $f = 0$. This means the induced map $\Hom(S_{|T|},X) \rightarrow \Hom(T[-1],X)$ is injective (and hence is a bijection by the definition of a left minimal $(\Filt S)$-approximation).
	
	
	(ii): Suppose there exists $\eta \in \Hom(S_{|T|},S[1]) = \Ext(S_{|T|},S)$ such that $\eta\circ g_{S,|T|}^+ = 0 \in \Hom(T[-1],S[1]) = \Ext(T[-1],S)$. We consider $\eta$ as a short exact sequence $S\hookrightarrow E\xrightarrowtwohead{q} S_{|T|}$ as shown below. Then $\eta\circ g_{S,|T|}^+=0$ means there exists $h:T[-1]\to E$ so that $q\circ h=g_{S,|T|}^+$. But $\Hom(S_{|T|},E)\cong\Hom(T[-1],E)$ by (i), so there is a unique morphism $h':S_{|T|}\to E$ so that $h'\circ g_{S,|T|}^+=h$. But then $q\circ h=q\circ h'\circ g_{S,|T|}^+=g_{S,|T|}^+$. Since $\Hom(T[-1],S_{|T|})\cong\Hom(S_{|T|},S_{|T|})$ by (i), we conclude that $q\circ h'$ must be the identity on $S_T$. Thus, $\eta=0$ (i.e., the corresponding short exact sequence is split).
	
	\begin{center}
		\begin{tikzcd}
		S \arrow[r, hook] \arrow[d, equal] & E' \arrow[d] \arrow[r, two heads] & T[-1] \arrow[d, "g_{S,|T|}^+"] \arrow[dl, dashed, "h" above]\\
		S \arrow[r, hook] & E \arrow[r, two heads, "q"] & S_{|T|} \arrow[l, dashed, bend left, "h'"]\\
		\end{tikzcd}
	\end{center}
\end{proof}

We are now ready to define mutation.

\begin{define}\label{def:mutation}\
	\begin{enumerate}[label=\upshape(\alph*)]
		\item \cite[Def. 1.8]{hanson_tau} Let $\S_p \sqcup \S_n[1]$ be a semibrick pair. We say that $\S_p \sqcup \S_n[1]$ is \emph{singly left mutation compatible} at $S \in \S_p$ if for all $T \in \S_n$, a left minimal $(\Filt S)$-approximation $g^+_{S,T}: T \rightarrow S_T$ is either a monomorphism or an epimorphism. We say $\S_p \sqcup \S_n[1]$ is singly left mutation compatible\footnote{In \cite{hanson_tau} the term \emph{mutation compatible semibrick pair} is used. We have reserved this terminology for a more subtle property.} if it is singly left mutation compatible at every $S \in \S_p$.
		\item Let $\X = \S_p \sqcup \S_n[1]$ be singly left mutation compatible at $S \in \S_p$. We define the \emph{left mutation} of $\X$ at $S$, denoted $\mu^+_{S,\X}$, as follows.
			\begin{itemize}
				\item $\mu^+_{S,\X}(S) := S[1]$.
				\item For $S \neq T \in \S_p,$ define $\mu^+_{S,\X}(T) := \cone(g^+_{S,T})$, where $g^+_{S,T}:T[-1]\rightarrow S_T$ is a left minimal $(\Filt S)$-approximation. In particular, there is an exact sequence $S_T\hookrightarrow \mu^+_{S,\X}(T) \twoheadrightarrow T$.
				\item For $T \in \S_n,$ define $\mu^+_{S,\X}(T[1]) := \cone(g^+_{S,T})$, where $g^+_{S,T}:T\rightarrow S_T$ is a left minimal $(\Filt S)$-approximation. In particular, if $g^+_{S,T}$ is mono, then $\mu^+_{S,\X}(T[1]) = \coker(g^+_{S,T})$ and if $g^+_{S,T}$ is epi, then $\mu^+_{S,\X}(T[1]) = \ker(g^+_{S,T})[1]$.
			\end{itemize}
		\item Let $\S_p \sqcup \S_n[1]$ be a semibrick pair. We say that $\S_p \sqcup \S_n[1]$ is \emph{singly right mutation compatible} at $S \in \S_n$ if for all $T \in \S_p$, a right minimal $(\Filt S)$-approximation $g^-_{S,T}: S_T \rightarrow T$ is either a monomorphism or an epimorphism. We say $\S_p \sqcup \S_n[1]$ is singly right mutation compatible if it is singly right mutation compatible at every $S \in \S_n$.
		\item Let $\X = \S_p \sqcup \S_n[1]$ be a singly right mutation compatible at $S \in \S_n$. We define the \emph{right mutation} of $\X$ at $S$, denoted $\mu^-_{S,\X}$, as follows.
			\begin{itemize}
				\item $\mu^-_{S,\X}(S[1]) := S$.
				\item For $S \neq T \in \S_n,$ define $\mu^-_{S,\X}(T[1]) := \cocone(g^-_{S,T})[1]$, where $g^-_{S,T}:S_T \rightarrow T[1]$ is a right minimal $(\Filt S)$-approximation. In particular, there is an exact sequence $T\hookrightarrow \mu^-_{S,\X}(T)[-1] \twoheadrightarrow S_T$.
				\item For $T \in \S_p,$ define $\mu^-_{S,\X}(T) := \cocone(g^-_{S,T})[1]$, where $g^-_{S,T}: S_T \rightarrow T$ is a right minimal $(\Filt S)$-approximation. In particular, if $g^-_{S,T}$ is mono, then $\mu^-_{S,\X}(T) = \coker(g_{S,T})^-$ and if $g_{S,T}^-$ is epi, then $\mu^-_{S,\X}(T) = \ker(g^-_{S,T})[1]$.
			\end{itemize}
	\end{enumerate}
\end{define}

\begin{rem}\label{rem:mono epi motivation}
The assumption that, for $T \in \S_n$, the map $g^+_{S,T}$ is mono or epi is equivalent to $\mu^+_{S,\X}(T) \in \mods\Lambda\sqcup\mods\Lambda[1]$.
\end{rem}

\begin{rem}\label{rem:mutationagrees}
	If $\X = \S_p \sqcup \S_n[1] \in \smc\Lambda$ then $\X$ is both singly left mutation compatible and singly right mutation compatible. In this case, the new definitions of mutation agree with those for 2-simple minded collections. Moreover, these definitions are pairwise in the following sense. Suppose $\X = \S_p\sqcup \S_n[1]$ and $\X' = \S'_p \sqcup \S'_n[1]$ are semibrick pairs with $\X \subset \X'$ and let $S \in \S'_p$. If $\X'$ (and hence $\X$) is singly left mutation compatible at $S$, then for $S \neq T \in \X$, we have $\mu_{S,\X}^+(T) = \mu_{S,\X'}^+(T)$.
\end{rem}

\begin{prop}\label{prop:mutation}
	Let $\X = \S_p \sqcup \S_n[1]$ be a semibrick pair which is singly left mutation compatible at $S \in \S_p$. Then
	\begin{enumerate}[label=\upshape(\alph*)]
		\item $\mu^+_{S,\X}(\X)$ is a semibrick pair which is singly right mutation compatible at $S$.
		\item $\mu^-_{S,\mu^+_{S,\X}(\X)}\circ\mu^+_{S,\X}(\X) = \X$.
	\end{enumerate}
	Likewise, the dual result holds for any singly right mutation compatible semibrick pair.
\end{prop}

The proof is similar to that of \cite[Prop. 7.6]{koenig_silting} and \cite[Thm. 6.2]{dugas_tilting}, but we include it here for completeness and to emphasize that the result does not depend on starting with a 2-simple minded collection.

\begin{proof}
	Let $\X = \S_p \sqcup \S_n[1]$ be a semibrick pair which is singly left mutation compatible at $S \in \S_p$. For $T\neq S \in \X$, let $T[-1]\xrightarrow{g^+_{S,|T|}}S_{|T|} \rightarrow \mu^+_{S,\X}(T) \rightarrow$ be the defining triangle for $\mu^+_{S,\X}(T)$.
	
	\emph{Claim 1: $\Hom(\mu^+_{S,\X}(S), \mu^+_{S,\X}(T)[m]) = 0$ for all $T \neq S \in \X$ and $m\in \{-1,0\}$}. Indeed, since $\mu^+_{S,\X}(S) = S[1] \in\mods\Lambda[1]$ and $\mu^+_{S,\X}(T) \in \mods\Lambda\sqcup \mods\Lambda[1]$, we have that $\Hom(\mu^+_{S,\X}(S), \mu^+_{S,\X}(T)[-1]) = 0$ automatically. Moreover, applying $\Hom(S[1],-)$ to the defining triangle for $\mu^+_{S,\X}(T)$, we have an exact sequence
	$$0 = \Hom(S[1], S_{|T|}) \rightarrow \Hom(S[1], \mu^+_{S,\X}(T)) \rightarrow \Hom(S[1], T[1]) = 0,$$
	which proves the claim.
	
	\emph{Claim 2: $\Hom(\mu^+_{S,\X}(T), \mu^+_{S,\X}(S)[m]) = 0$ for all $T \neq S \in \X$ and $m \in \{-1,0\}$.} Indeed, applying $\Hom(-, S[-1])$ to the defining triangle for $\mu^+_{S,\X}(T)$, we have an exact sequence
	$$0 = \Hom(T, S[-1]) \rightarrow \Hom(\mu^+_{S,\X}(T), S[-1]) \rightarrow \Hom(S_{|T|}, S[-1])= 0.$$
	Moreover, by Lemma \ref{lem:approxconditions}, we have an exact sequence
	\begin{eqnarray*}
		&&0= \Hom(T[-1], S[-1]) \rightarrow \Hom(\mu^+_{S,\X}(T), S)\xrightarrow{0}\Hom(S_{|T|},S)\rightarrow\Hom(T[-1],S)\xrightarrow{0}\\
		&&\xrightarrow{0} \Hom(\mu^+_{S,\X}(T), S[1]) \xrightarrow{0} \Hom(S_{|T|}, S[1]) \rightarrow \Hom(T[-1],S[1])
	\end{eqnarray*}
	This proves the claim since $\mu^+_{S,\X}(S) = S[1]$.
	
	\emph{Claim 3: $\Hom(\mu^+_{S,\X}(T), \mu^+_{S,\X}(T')[m]) \cong \Hom(T, T'[m])$ for $T, T' \neq S \in \X$} and $m \in \{-1,0\}$. Indeed, applying $\Hom(\mu^+_{S,\X}(T), -)$ to the defining triangle of $\mu^+_{S,\X}(T')$, we have an exact sequence
	\begin{eqnarray*}
		0 &=& \Hom(\mu^+_{S,\X}(T), S_{|T'|}[m]) \rightarrow \Hom(\mu^+_{S,\X}(T), \mu^+_{S,\X}(T')[m]) \rightarrow\\
		&\rightarrow& \Hom(\mu^+_{S,\X}(T), T'[m]) \rightarrow \Hom(\mu^+_{S,\X}(T), S_{|T'|}[m+1]) = 0
	\end{eqnarray*}
	for $m \in \{-1,0\}$ by Claim 2. Likewise, applying $\Hom(-, T')$ to the defining triangle of $\mu^+_{S,\X}(T)$, we have an exact sequence
	\begin{eqnarray*}
		0 &=& \Hom(S_{|T|}, T'[m-1]) \rightarrow \Hom(T[-1], T'[m-1]) \rightarrow\\
		&\rightarrow& \Hom(\mu^+_{S,\X}(T), T'[m]) \rightarrow \Hom(S_{|T|}, T'[m]) = 0
	\end{eqnarray*}
	for $m \in \{-1,0\}$. This proves Claim 3. In particular, taking $T = T'$, we have that $\mu^+_{S,\X}(T)$ is a brick.
	
	Remark \ref{rem:mono epi motivation}, together with Claims 1, 2, and 3, implies that $\mu^+_{S,\X}(\X)$ is a semibrick pair. To show that $\mu^+_{S,\X}(\X)$ is singly right mutation compatible at $S$ and $\mu^-_{S,\mu^+_{S,\X}(\X)}\circ\mu^+_{S,\X}(\X) = \X$, we observe that in the triangle $T[-1]\xrightarrow{g^+_{S,|T|}}S_{|T|} \xrightarrow{g^-_{S,|\mu^+_{S,\X}(T)|}} \mu^+_{S,\X}(T) \rightarrow$, the map $g^+_{S,|T|}$ is a left minimal ($\Filt S)$-approximation with cone $\mu^+_{S,\X}(T)$ if and only if the map $g^-_{S,|\mu^+_{S,\X}(T)|}$ is a right minimal $(\Filt S)$-approximation with cocone $T[-1]$.
\end{proof}

Before we propose our new definition of mutation compatibility for a semibrick pair, we need the following results.

\begin{lem}\label{lem:reachzero}\
	\begin{enumerate}[label=\upshape(\alph*)]
		\item Let $\X = \S_p \sqcup \S_n[1]$ be singly left mutation compatible at $S \in \S_p$. Write $\mu^+_{S,\X}(\X) = \S'_p \sqcup \S'_n[1]$. Then $\Filt\Fac(\S'_p) \subsetneq \Filt\Fac(\S_p)$.
		\item Let $\X = \S_p\sqcup \S_n[1]$ be a semibrick pair. If $\X$ is singly left mutation compatible and $\S_p \neq \emptyset$, choose $S \in \S_p$ and let $\X' = \mu^+_{S,\X}(\X)$. Repeat this process for $\X'$. Then after finitely many mutations, we reach a semibrick pair $\Y$ for which either $\Y$ is not singly left mutation compatible or $\Y = \S[1]$ for some $\S \in \sbrick\Lambda$.
	\end{enumerate}
\end{lem}

\begin{proof}
(a) Let $T \in \S'_p$. By Definition \ref{def:mutation}(b), there are two possibilities. If $T = \mu^+_{S,\X}(T')$ with $T' \in \S_p$, then there exists $S_{T'} \in \Filt S$ and an exact sequence $S_T\hookrightarrow T \twoheadrightarrow T'$. Otherwise, $T = \mu^+_{S,\X}(T'[1])$ with $T' \in \S_n$ and there exists $S_{T'} \in \Filt S$ and an exact sequence $T'\hookrightarrow S_{T'} \twoheadrightarrow T$ (note that in this case, the map $g_{S,T'}^+$ must be mono). In either case, we have $T \in \Filt\Fac(\S_p)$ and hence $\Filt\Fac(\S'_p) \subseteq \Filt\Fac(\S_p)$. Moreover, we observe that $S \in \Filt\Fac(\S_p) \setminus \Filt\Fac(\S'_p)$.

(b) If this process did not terminate, we would end up with an infinite descending chain of torsion classes by (a). This violates the assumption that $\Lambda$ is $\tau$-tilting finite.
\end{proof}

We now propose the following definition.

\begin{define}\label{def:mutcompat}
	Let $\X = \S_p \sqcup \S_n[1]$ be a semibrick pair. We call $\X$ \emph{mutation compatible} if any of the following hold:
	\begin{enumerate}[label=\upshape(\alph*)]
		\item $\S_p = \emptyset$, that is, $\X$ is a shifted semibrick.
		\item $\S_n = \emptyset$, that is, $\X$ is a semibrick.
		\item $\X$ is singly left mutation compatible and there exists $S \in \S_p$ such that $\mu^+_{S, \X}(\X)$ is a mutation compatible semibrick pair.
	\end{enumerate}
\end{define}

\begin{rem}\label{rem:reachzero}\
	\begin{enumerate}[label=\upshape(\alph*)]
		\item This notion is well-defined by the previous lemma.
		\item We could have chosen to define left mutation compatible and right mutation compatible separately. The reason we chose not to do so is that these conditions turn out to be equivalent as a corollary of the following theorem.
	\end{enumerate}
\end{rem}

We are now ready to prove our first main theorem (Theorem A in the introduction).

\begin{thm}\label{thm:mutationcomplete}
	Let $\X = \S_p \sqcup \S_n[1]$ be a semibrick pair. Then $\X$ is completable if and only if $\X$ is mutation compatible.
\end{thm}

\begin{proof}
	Let $\X = \S_p \sqcup \S_n[1]$ be a mutation compatible semibrick pair. We already know that $\X$ is completable if $\S_p = \emptyset$ or $\S_n = \emptyset$ (Theorem \ref{thm:semibrickcompletable}). Otherwise, by Definition \ref{def:mutcompat}, there exists a sequence of left mutations
	$$\X = \X_0 \xrightarrow{\mu^+_{S_0, \X_0}} \X_1 \xrightarrow{\mu^+_{S_1,\X_1}} \X_2 \xrightarrow{\mu^+_{S_2,\X_2}}\cdots\xrightarrow{\mu^+_{S_{m-1},\X_{m-1}}} \X_m$$
	such that $\X_m$ is a shifted semibrick. It follows that $\X_m$ is completable (say to $\Y_m \in \smc\Lambda$). Thus by the pairwise nature of mutation (see Remark \ref{rem:mutationagrees}), there exists a sequence of right mutations of 2-simple minded collections
	$$\X \subset \Y_0 \xleftarrow{\mu^-_{S_0,\Y_1}}\Y_1\xleftarrow{\mu^-_{S_1,\Y_2}} \Y_2 \xleftarrow{\mu^-_{S_2,\Y_3}}\cdots \xleftarrow{\mu^-_{S_{m-1},\Y_m}} \Y_m$$
	and therefore $\X$ is completable.
	
	Now let $\X = \S_p \sqcup \S_n[1]$ be a completable semibrick pair, so $\X \subset \Y$ for some $\Y \in \smc\Lambda$. If $\S_p$ or $\S_n$ is empty, then $\X$ is mutation compatible, so assume both are nonempty. We now define a sequence of semibrick pairs and 2-simple minded collections as follows.
	\begin{itemize}
		\item $\S_{0,p} \sqcup \S_{0,n}[1] = \X_0 := \X$ and $\Y_0 := \Y$.
		\item If $\S_{i-1,p} \neq \emptyset$, then choose $S_{i-1} \in \S_{i-1,p}$ and define $\S_{i,p}\sqcup\S_{i,n}[1] = \X_i:= \mu^+_{S_{i-1},\Y_{i-1}}(\X_{i-1})$ and $\Y_i:= \mu^+_{S_{i-1},\Y_{i-1}}(\Y_{i-1})$.
	\end{itemize}
	It follows that each $\X_i \subset \Y_i$ is contained in a 2-simple minded collection, and hence is singly left mutation compatible. Moreover, there exists some $\X_m$ which is a shifted semibrick by Remark \ref{rem:reachzero}. Therefore  by the pairwise nature of mutation (see Remark \ref{rem:mutationagrees}), $\X$ is mutation compatible.
\end{proof}

\begin{cor}\label{cor:mutationcompatiblepassdown}
	Let $\X = \S_p \sqcup \S_n[1]$ be mutation compatible. Then for $S \in \S_p$, the semibrick pair $\mu^+_{S,\X}(\X)$ is mutation compatible.
\end{cor}

\begin{proof}
	If $\X = \S_p \sqcup \S_n[1]$ is contained in the 2-simple minded collection $\Y$, then $\mu^+_{S,\X}(\X)$ is contained in the 2-simple minded collection $\mu^+_{S,\Y}(\Y)$.
\end{proof}

\subsection{The Pairwise 2-Simple Minded Compatibility Property}\label{sec:pairwiseproperty}

An interesting problem is to determine when the 2-simple minded collections of an algebra are given by pairwise conditions. More precisely, we are interested in which algebras satisfy the following definition.

\begin{define}\label{def:pairwise}
	Let $\X = \S_p \sqcup \S_n[1]$ be a semibrick pair. We say that $\X$ has the \emph{pairwise 2-simple minded compatibility property} if either $\X$ is completable or there exists $S \in \S_p, T \in \S_n$ such that $S\sqcup T[1]$, considered as a semibrick pair, is not completable. We say the algebra $\Lambda$ has the pairwise 2-simple minded compatibility property if every semibrick pair for $\mods\Lambda$ does.
\end{define}

Note that the negation of this property for $\mods\Lambda$ is: There exists a semibrick pair which is not completable, but in which any pair of direct summands can be completed to a 2-simple minded collection.

In Section \ref{sec:gentleclassification}, we will show that not every $\tau$-tilting finite algebra has the pairwise 2-simple minded compatibility property, disproving Conjecture 1.11 from \cite{hanson_tau}. We end this subsection by using our previous results to rephrase Definition \ref{def:pairwise} in terms of mutation compatibility.

\begin{prop}\label{prop:pairwise}
	Let $\Lambda$ be a $\tau$-tilting finite algebra. Then $\Lambda$ has the pairwise 2-simple minded compatibility property if and only if for every singly left mutation compatible semibrick pair $\X = \S_p \sqcup \S_n[1]$, the following are equivalent.
	\begin{enumerate}[label=\upshape(\roman*)]
		\item For all $S \in \S_p, T \in \S_n$, the semibrick pair $S\sqcup T[1]$ is mutation compatible.
		\item $\X$ is mutation compatible.
	\end{enumerate}
\end{prop}

\begin{proof}
	This follows directly from Theorem \ref{thm:mutationcomplete}.
\end{proof}

\begin{rem}
It may seem that the pairwise compatibility property should imply that every singly left mutation compatible semibrick pair is mutation compatible. However, a priori, the condition that $S \sqcup T[1]$ be mutation compatible is necessary. Indeed, assume that $S \sqcup T[1]$ is singly left mutation compatible and let $g^+_{S,T}: T \rightarrow S_T$ be a left minimal $(\Filt S)$-approximation. If $g^+_{S,T}$ is epi, then $\mu^+_{S, S\sqcup T[1]}(S\sqcup T[1]) = S[1] \sqcup \ker(g^+_{S,T})[1]$ and $S\sqcup T[1]$ is mutation compatible. However, if $g^+_{S,T}$ is mono, then $\mu^+_{S, S\sqcup T[1]}(S\sqcup T[1]) = \coker(g^+_{S,T})\sqcup S[1]$, but it does not immediately follow that the left minimal $(\Filt \coker(g^+_{S,T}))$-approximation $S \rightarrow \coker(g^+_{S,T})_S$ is mono or epi. We do not, however, have an example where this is not the case.
\end{rem}

We can further refine Proposition \ref{prop:pairwise} for algebras with particularly well-behaved approximations, as shown in the following.

\begin{prop}\label{prop:pairwise2}
	Let $\Lambda$ be a $\tau$-tilting finite algebra such that for all $S,T\in\brick\Lambda$, we have the following.
		\begin{enumerate} [label=\upshape(\alph*)]
			\item If there exists a morphism $T \rightarrow S$ which is mono or epi, then it is a left minimal $(\Filt S)$-approximation. In particular, $\dim\Hom(T,S) = 1$.
			\item If there does not exist a morphism $T\rightarrow S$ which is mono or epi, then there does not exist a nonzero left minimal $(\Filt S)$-approximation $T \rightarrow S_T$ which is mono or epi.
		\end{enumerate}
	Then the following are equivalent.
	\begin{enumerate}[label=\upshape(\roman*)]
		\item $\Lambda$ has the pairwise 2-simple minded compatibility property.
		\item Every singly left mutation compatible semibrick pair is mutation compatible.
		\item For every singly left mutation compatible semibrick pair $\X = \S_p \sqcup \S_n[1]$ and $S \in \S_p$, the semibrick pair $\mu^+_{S, \X}(\X)$ is singly left mutation compatible.
	\end{enumerate}
\end{prop}

\begin{proof}
	Let $\Lambda$ be such an algebra and let $\X = \S_p \sqcup \S_n[1]$ be a singly left mutation compatible semibrick pair. Let $S \in \S_p$ and $T \in \S_n$. We claim that $S \sqcup T[1]$ is completable. Indeed, if $\Hom(T,S) = 0$, then $\mu^+_{S, S\sqcup T[1]}(S\sqcup T[1]) = S[1] \sqcup T[1]$ and hence $S \sqcup T[1]$ is mutation compatible. If $\Hom(T,S) \neq 0$, then by (a) and (b), we have that $\dim\Hom(T,S) = 1$ and there is a morphism $g^+_{S,T}: T \rightarrow S$ which is mono or epi and is a left minimal $(\Filt S)$-approximation. If $g^+_{S,T}$ is epi, then $\mu^+_{S, S\sqcup T[1]}(S\sqcup T[1]) = S[1] \sqcup \ker(g^+_{S,T})[1]$ and hence $S \sqcup T[1]$ is mutation compatible. If $g^+_{S,T}$ is mono, then $\mu^{+}_{S, S\sqcup T[1]}(S\sqcup T[1]) = \coker(g^+_{S,T}) \sqcup S[1]$ and the quotient map $S \twoheadrightarrow \coker(g_{S,T})$ is a left minimal $(\Filt \coker(g^+_{S,T}))$-approximation by (a). This means $\coker(g^+_{S,T})\sqcup S[1]$ is singly left mutation compatible and $\mu^+_{\coker(g^+_{S,T}),\coker(g^+_{S,T})\sqcup S[1]}(\coker(g^+_{S,T})\sqcup S[1]) = T[1] \sqcup \coker(g^+_{S,T})[1]$; hence, $S\sqcup T[1]$ is mutation compatible. This shows that every singly left mutation compatible semibrick pair of rank 2 is mutation compatible. Given this, the equivalence between (i) and (ii) is immediate from Proposition \ref{prop:pairwise}.
	
	To see that (iii) implies (ii), let $\X = \S_p \sqcup \S_n[1]$ be singly left mutation compatible. If $\S_p = \emptyset$, we are done. Otherwise, choose $S \in \S_p$ and let $\X' = \mu_{S,\X}^+(\X)$. By (iii), $\X'$ is singly left mutation compatible, so we can repeat this process. Since this process must terminate by Lemma \ref{lem:reachzero}(b), we conclude that $\X$ is mutation compatible.
	
	Finally, it follows from Corollary \ref{cor:mutationcompatiblepassdown} that (ii) implies (iii). Therefore, the three conditions are equivalent as claimed.
\end{proof}

The hypotheses of Proposition \ref{prop:pairwise2} include the simpler case where $\dim\Hom(T,S) \leq 1$ for all $S,T\in\brick\Lambda$ and bricks do not have nontrivial self-extensions. We will, however, need the weaker hypotheses when we discuss Nakayama-like algebras in Section \ref{sec:cyclics}.

We conclude this section with a new proof that every representation finite hereditary algebra has the pairwise 2-simple minded compatibility condition. This result has already been shown by the second author and Todorov in \cite{igusa_signed} using the correspondence between bricks and $c$-vectors. We give here a proof that does not rely on this machinery.

\begin{prop}\label{prop:hereditary}
	Let $\Lambda$ be hereditary. Then every semibrick pair is singly left mutation compatible and $\Lambda$ has the 2-simple minded compatibility property.
\end{prop}

\begin{proof}
	Let $\S_p \sqcup \S_n[1]$ be a semibrick pair. Let $S \in \S_p, T \in \S_n$, and let $g^+_{S,T}: T\rightarrow S_T$ be a left minimal  $(\Filt S)$-approximation. It follows from the proof of Proposition \ref{prop:mutation} that $\cone(g^+_{S,T})$ is a brick. Moreover, it is well known that since $\Lambda$ is hereditary, $\cone(g^+_{S,T}) \cong \coker(g_{S,T}^+)\sqcup \ker(g^+_{S,T})[1]$. Therefore, $g^+_{S,T}$ is either mono or epi. We conclude that $\S_p \sqcup \S_n[1]$ is singly left mutation compatible. As in the proof of Proposition \ref{prop:pairwise2} above, this implies that $\Lambda$ has the pairwise 2-simple minded compatibility property.
\end{proof}

\section{Pairwise 2-Simple Minded Compatibility for Nakayama-like Algebras}\label{sec:cyclics}

The goal of this section is to show that all \emph{Nakayma-like algebras} have the pairwise 2-simple minded compatibility property.

\begin{define}\label{def:nakayamalike}
	Consider the quivers
	\begin{center}
		\begin{tikzpicture}
			\node at (0,0) {1};
			\node at (1,0) {2};
			\node at (3.5,0) {$n$};
			\draw (.2,0)--(.8,0) node[midway,anchor=north]{$\gamma_1$};
			\draw (1.2,0)--(1.8,0) node[midway,anchor=north]{$\gamma_2$};
			\node at (2.25,0) {$\cdots$};
			\draw (2.7,0)--(3.3,0) node[midway,anchor=north]{$\gamma_{n-1}$};
			\draw (3.5,.2)--(3.5,.5)--(0,.5)--(0,.2);
			\draw (0,.5)--(3.5,.5) node[midway,anchor=south]{$\gamma_n$};
			\node at (1.75,-1){$\Delta_n$};
		\begin{scope}[shift = {(5,0)}]
			\node at (0,0) {1};
			\node at (1,0) {2};
			\node at (3.5,0) {$n$};
			\draw (.2,0)--(.8,0) node[midway,anchor=north]{$\gamma_1$};
			\draw (1.2,0)--(1.8,0) node[midway,anchor=north]{$\gamma_2$};
			\node at (2.25,0) {$\cdots$};
			\draw (2.7,0)--(3.3,0) node[midway,anchor=north]{$\gamma_{n-1}$};
			\node at (1.75,-1){$A_n$};
		\end{scope}
		\end{tikzpicture}
	\end{center}
	where each $\gamma_i$ has arbitrary orientation. By a \emph{Nakayama-like algebra}, we mean an algebra of the form $K\Delta_n/I$ or $KA_n/I$ where $I$ is an admissible ideal generated by monomials and $\Lambda \ncong K\widetilde{A}_n$ for some orientation of $\widetilde{A}_n$ (the extended Dynkin diagram of type $A_n$).
\end{define}

\begin{rem}\label{rem:othercyclic}
	In the case that our algebra is of the form $KA_n/I$, we remark that the condition that $I$ be generated by monomials is superfluous. This is not the case for algebras of the form $K\Delta_n/I$, where the resulting algebra is either Nakayama-like, extended Dynkin, or of the form $KQ_{i,j}/I$, where $Q_{i,j}$ is the quiver
	\begin{center}
	\begin{tikzcd}[row sep = tiny]
		& n \arrow[r,"\beta_2"] & \cdots \arrow[r,"\beta_{j-1}"] & i+2 \arrow[dr,"\beta_j"]\\
		1 \arrow[ur,"\beta_1"] \arrow[dr,"\alpha_1"]&&&& i+1\\
		& 2 \arrow[r,"\alpha_2"] & \cdots \arrow[r,"\alpha_{i-1}"] & i \arrow[ur,"\alpha_i"]
	\end{tikzcd}
	\end{center}
and $I = (\alpha_1\cdots\alpha_i \pm \beta_1\cdots\beta_j)$, where the two signs are equivalent.
\end{rem}

The name Nakayama-like comes from the fact that such an algebra is Nakayama if and only if each arrow is oriented the same direction (say $i\rightarrow i+1$). We remark that we can consider Nakayama-like algebras with quiver $A_n$ to also be quotients of $K\Delta_n$ by relaxing the condition that $I$ be admissible.

The goal of this section is to prove the following theorem (Theorem B in the introduction).

\begin{thm}\label{thm:cyclichaveproperty}
	Let $\Lambda$ be a Nakayama-like algebra. Then $\Lambda$ is $\tau$-tilting finite and has the pairwise 2-simple minded compatibility property.
\end{thm}

Nakayama algebras themselves have been shown to have the pairwise 2-simple minded compatibility property in \cite{hanson_tau} and in \cite{igusa_category}, \cite{igusa_signed} when $\Lambda \cong KA_n$. Algebras of the form $K \widetilde{A}_n$ are known to be $\tau$-tilting infinite, which is why they are excluded from the definition of Nakayama-like algebras and the statement of Theorem \ref{thm:cyclichaveproperty}. The fact that Nakayama-like algebras are $\tau$-tilting finite is immediate, since they are representation-finite algebras.


For the remainder of this section, we fix a Nakayama-like algebra $\Lambda$. We begin by giving a description of the indecomposable $\Lambda$-modules and bricks, and the morphisms and extensions between them. This description is similar to the description for Nakayama algebras in \cite[Section 3.1]{hanson_tau} or \cite{adachi_classification}, and differs slightly from the standard description in terms of strings.

Recall that given a quiver with relations $Q/I$, there exists another quiver with relations $\overline{Q}/\overline{I}$ called the \emph{universal cover} of $Q/I$. Readers are referred to \cite{martinez_universal} for details.

We begin by writing $\Lambda = KQ/I$ where $Q \in \{A_n,\Delta_n\}$. If $Q = A_n$, then $A_n/I$ is its own universal cover. Otherwise, the universal cover of $Q/I$ is of the form $A_\infty/J$, where $A_\infty$ is a quiver whose underlying graph has vertex set $\Z$ and contains an edge between $i$ and $j$ if and only if $j = i+1$.

We now fix some notation that we will use to construct our model. We denote by $i_n$ the unique integer $1\leq k \leq n$ such that $i-k\in n\Z$. We then define the ordered multisets
$$[i,j]_n := \begin{cases}
	[i_n,j_n] \textnormal{ if }i_n < j_n\\
	[i_n,n]\cup[1,j_n] \textnormal{ if } i_n \geq j_n
\end{cases}$$
and likewise for $(i,j)_n$, etc. For example, $(1,4)_5 = \{2,3\}$ and $(4,1)_5 = \{5\}$. Similarly, we define 
$$[i,i,j]_n := [i,i]_n \cup (i,j]_n.$$ For example, $[4,4,1]_5 = \{4, 5, 1, 2, 3, 4, 5, 1\}$. Lastly, given three marked points $i,j,k$, we say $i < _n j <_n k$ if $j \in (i,k)_n$. Equivalently, either $i, j$, and $k$ are distinct and in cyclic order or $i = k \neq j$.

Alternatively, if we fix a section of the covering map $A_\infty \rightarrow \Delta_n$ we can label each $x \in (A_\infty)_0$ with $i_j$ where $1 \leq i \leq n$ and $j \in \Z$ (that is, $x$ corresponds to the $j$-th copy of the vertex $i$). Thus considering $\Z$ with this labeling, we can define, for example
$$[i,j]_n := \begin{cases}
	[i_0,j_0] \textnormal{ if }i < j\\
	[i_0,n_0]\cup[1_1,j_1] \textnormal{ if } i \geq j
\end{cases}$$
for all $1 \leq i,j \leq n$. We can likewise define intervals of the form $[i,i,j]_n$ in this way.

We are now ready to describe the indecomposable $\Lambda$-modules. For $M \in \mods\Lambda$, we denote the length of $M$ by $l(M)$. The following is immediate from the classification of (isomorphism classes of) indecomposable modules as (equivalence classes) of strings.
\begin{prop}
	For every $i \leq n$, there is an integer $l(i)$ such that for $j \leq l(i)$ there exists a unique string, either equal to $e_i$ or starting with $e_i\gamma_i^{\epsilon_i}$, such that the corresponding module has length $j$. Moreover, all (equivalence classes) of strings appear in this way. Thus there is a bijection
	$$M:\{(i,j)\in\Z^2|0 < i \leq n,0 < j \leq l(i)\}\leftrightarrow\ind(\mods\Lambda)$$
\end{prop}

We say a multiset of the form $[i,j]_n$ (resp. $[i,i,j]_n$) \emph{contains a relation} if $(j-i)_n > l(i)$ (resp. $n + (j-i)_n > l(i)$).

We are now ready to construct our combinatorial model. Let $D(n,l(1),l(2),\ldots,l(n))$ be the punctured disk $D^2\setminus\{0\}$ with $n$ marked points on its boundary, labeled counterclockwise $1,2,\ldots,n$. We label the marked point $i$ with a solid circle if $\gamma_{(i-1)_n}$ is oriented from $(i-1)_n$ to $i$ and with a solid square if $\gamma_{(i-1)_n}$ is oriented from $i$ to $(i-1)_n$. If it is unknown or irrelevant which way $\gamma_{(i-1)_n}$ is oriented, we label the marked point $i$ with a hollow circle\footnote{We also label $1$ with a hallow circle in the case that $Q = A_n$ and the arrow $\gamma_n$ does not exist.}. The value of $l(i)$ is as given in the above proposition and is called the \emph{length} of the marked point labeled $i$.

We now wish to relate $\brick\Lambda$ to certain directed paths in $D(n,l(1),\ldots,l(n))$. We first observe the following.

\begin{lem}\label{lem:bricks}
	Let $M(i,j) \in \ind(\mods\Lambda)$. Then $M(i,j)$ is a brick if and only if $j \leq n$. Moreover, in this case, $\End(M(i,j)) \cong K$.
	\begin{proof}
		It is clear that $M(i,j)$ is a brick if $j \leq n$, thus suppose $j > n$. Denote $j':=(i + j)_n$. If $i$ and $j'$ are both circles or both squares, then as in \cite[Prop. 3.4]{hanson_tau}, we have a chain or morphisms $$M(i, j) \twoheadrightarrow M(i, (j'-i)_n) \hookrightarrow M(i,j)$$
		and hence $M(i,j)$ is not a brick. Thus assume without loss of generality that $i$ is a circle and $j'$ is a square. It follows that either $(i,i)_n$ or $(j',j')_n$ must contain a relation, contradicting that $j \leq l(i)$.
		
		Now let $M(i,j)$ be a brick. We wish to show that $\End(M(i,j)) \cong K$. We observe that since $j \leq n$, the linear transformation corresponding to every arrow in the (nonempty) set $\{\gamma_{(i+j)_n}\ldots,\gamma_{(i-1)_n}\}$ is the zero map. Thus we can consider $M(i,j)$ as a representation of the subquiver containing only the arrows $\{\gamma_i\ldots\gamma_{(i+j-1)_n}\}$, which is of type $A_n$. The result then follows immediately.
	\end{proof}
\end{lem}

In light of Lemma \ref{lem:bricks}, we propose the following definition.

\begin{define}
	A directed path $a: i\rightarrow j$ between two marked points of $D(n,l(1),\ldots,l(n))$ is called an \emph{arc} if
\begin{enumerate}[label=\upshape(\alph*)]
	\item $a$ is homotopic to the counterclockwise boundary arc $i\rightarrow j$.
	\item $a$ does not intersect itself unless $i=j$, in which case the only intersection occurs at the endpoint.
	\item $l(i)\geq (j-i)_n$.
\end{enumerate}
We call $i$ the \emph{source} of $a$, denoted $s(a)$, and $j$ the \emph{target} of $a$, denoted $t(a)$. We call $(j-i)_n$ the \emph{length} of $a$, denoted $l(a)$, and $[s(a),t(a)]_n$ the \emph{support} of $a$, denoted $\supp(a)$. Condition (3) can then be rephrased as $l(s(a)) \geq l(a)$. We denote the set of homotopy classes of arcs of $D(n,l(1),\ldots l(n))$ by $\Arc(n,l(1),\ldots,l(n))$.
\end{define}

The following is automatic from the definition and Lemma \ref{lem:bricks}.

\begin{prop}
There is a bijection $M: \Arc(n,l(1),\ldots,l(n)) \rightarrow \brick\Lambda$ given by sending an arc $a$ to the module $M(s(a),l(a))$.
\end{prop}

\begin{ex}
	Consider the algebra $\Lambda = KQ/I$ where $Q$ is the quiver
	\begin{center}
		\begin{tikzcd}
			1 \arrow[r, swap, "\beta"] \arrow[rr, bend left, "\gamma"] & 2 \arrow[r, swap, "\alpha"] & 3 \\
		\end{tikzcd}
	\end{center}
	and $I = (\beta\alpha)$. Then $\brick\Lambda$ corresponds to $\Arc(3, 2, 4, 3)$, as partially shown below, where we have labeled each arc with the corresponding brick written in the form $M(-,-)$ and as a string.
	
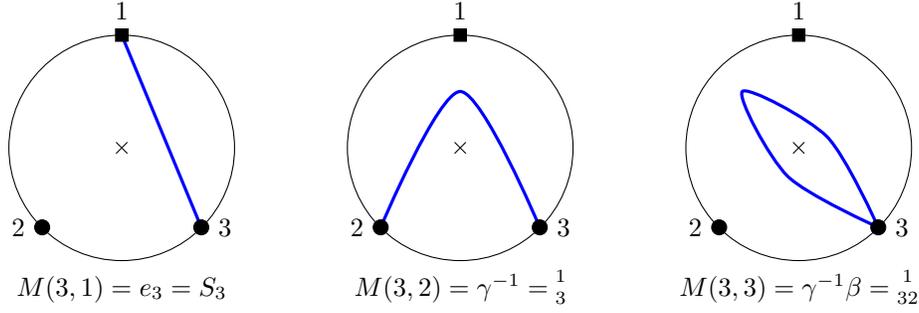
\begin{center}
\begin{tikzpicture}[scale=0.75]
	\draw (0,0) circle[radius=2cm];
	\draw (-.1,-.1) to (.1,.1);
	\draw (-.1,.1) to (.1,-.1);
		\draw[very thick,blue] plot [smooth] coordinates{(1.41,-1.41)(0,2)};
		\node at (1.41,-1.41) [draw, fill,circle,scale=0.6,label=east:3]{};
		\node at (-1.41,-1.41) [draw,fill,circle,scale=0.6, label=west:2]{};
		\node at (0,2) [draw,fill,rectangle,scale=0.7,label=north:1]{};
		\node at (0,-2.5) {$M(3, 1) = e_3 = S_3$};
\begin{scope}[shift={(6,0)}]
	\draw (0,0) circle[radius=2cm];
	\draw (-.1,-.1) to (.1,.1);
	\draw (-.1,.1) to (.1,-.1);
		\draw[very thick,blue] plot [smooth] coordinates{(1.41,-1.41)(0,1)(-1.41,-1.41)};
		\node at (1.41,-1.41) [draw, fill,circle,scale=0.6,label=east:3]{};
		\node at (-1.41,-1.41) [draw,fill,circle,scale=0.6, label=west:2]{};
		\node at (0,2) [draw,fill,rectangle,scale=0.7,label=north:1]{};
		\node at (0,-2.5) {$M(3, 2) = \gamma^{-1} = {\scriptsize\begin{matrix}1\\3\end{matrix}}$};
\end{scope}
\begin{scope}[shift={(12,0)}]
	\draw (0,0) circle[radius=2cm];
	\draw (-.1,-.1) to (.1,.1);
	\draw (-.1,.1) to (.1,-.1);
		\draw[very thick,blue] plot [smooth] coordinates{(1.41,-1.41)(-.2,-.5)(-1,1)(.5,.2)(1.41,-1.41)};
		\node at (1.41,-1.41) [draw, fill,circle,scale=0.6,label=east:3]{};
		\node at (-1.41,-1.41) [draw,fill,circle,scale=0.6, label=west:2]{};
		\node at (0,2) [draw,fill,rectangle,scale=0.7,label=north:1]{};
		\node at (0,-2.5) {$M(3, 3) = \gamma^{-1}\beta = {\scriptsize \begin{matrix}1\\32\end{matrix}}$};
\end{scope}
\end{tikzpicture}
\end{center}
\end{ex}

We now wish to give an overview of the morphisms and extensions between bricks in terms of how the corresponding arcs intersect. We remark that the pictures following the statements of the lemmas give examples of each type of intersection only. In each picture, we draw $S$ in solid blue and $T$ in dashed orange.

\begin{lem}\label{lem:indecomposableExts}
	Let $S,T \in \brick\Lambda$. Then there is an exact sequence $T\hookrightarrow E \twoheadrightarrow S$ with $E$ indecomposable if and only if one of the following holds.
	\begin{enumerate}[label=\upshape(\roman*)]
		\item $t(S) = s(T)$ is a circle and $[s(S),t(T)]_n$ (resp. $[s(S),s(S),t(T)]_n$ if $l(S) + l(T) > n$) does not contain a relation. In this case, we have $E \cong M(s(S),(t(T)-s(S))_n)$ if $l(S) + l(T) \leq n$ and $E \cong M(s(S),n + (t(T)-s(S))_n)$ if $l(S) + l(T) > n$.
		\item $s(S) = t(T)$ is a square and $[s(T),t(S)]_n$ (resp. $[s(T),s(T),t(S)]_n$ if $l(S) + l(T) > n$) does not contain a relation. In this case, we have $E \cong M(s(T),(t(S)-s(T))_n)$ if $l(S) + l(T) \leq n$ and $E \cong M(s(T),(n + (t(S)-s(T))_n))$ if $l(S) + l(T) > n$.
	\end{enumerate}
\end{lem}

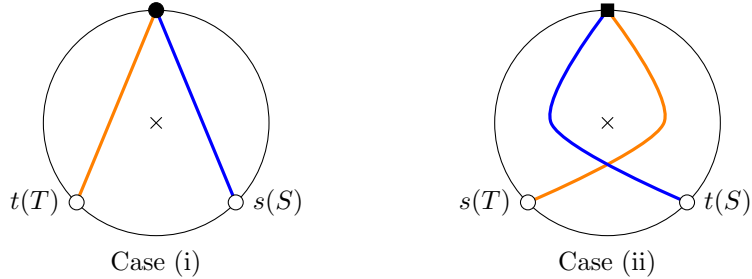
\begin{center}
\begin{tikzpicture}[scale=0.75]
	\draw (0,0) circle[radius=2cm];
	\draw (-.1,-.1) to (.1,.1);
	\draw (-.1,.1) to (.1,-.1);
		\draw[very thick,blue] plot [smooth] coordinates{(1.41,-1.41)(0,2)};
		\draw[very thick,orange,dashed] plot [smooth] coordinates{(-1.41,-1.41)(0,2)};
		\node at (1.41,-1.41) [draw, fill=white,circle,scale=0.6,label=east:$s(S)$]{};
		\node at (-1.41,-1.41) [draw,fill=white,circle,scale=0.6, label=west:$t(T)$]{};
		\node at (0,2) [draw,fill,circle,scale=0.6]{};
	\node at (0,-2.5) {Case (i)};
\begin{scope}[shift={(8,0)}]
	\draw (0,0) circle[radius=2cm];
	\draw (-.1,-.1) to (.1,.1);
	\draw (-.1,.1) to (.1,-.1);
		\draw[very thick,orange,dashed] plot [smooth] coordinates{(-1.41,-1.41)(1,0)(0,2)};
		\draw[very thick,blue] plot [smooth] coordinates{(1.41,-1.41)(-1,0)(0,2)};
		\node at (1.41,-1.41) [draw, fill=white,circle,scale=0.6,label=east:$t(S)$]{};
		\node at (-1.41,-1.41) [draw,fill=white,circle,scale=0.6, label=west:$s(T)$]{};
		\node at (0,2) [draw,fill,rectangle,scale=0.7]{};
		\node at (0,-2.5) {Case (ii)};
\end{scope}
\end{tikzpicture}
\end{center}

\begin{proof}
	In order for $E$ to be indecomposable, there must be an arrow from one of the ends of $T$ to one of the ends of $S$ (considering $S$ and $T$ as strings). This captures the two possibilities of this happening. We now consider three cases.
	
	If $l(S) + l(T) < n$, then we cannot have both $t(S) = s(T)$ and $s(S) = t(T)$. Thus assume without loss of generality that $t(S) = s(T)$. It is clear that any $E$ with an exact sequence $T \hookrightarrow E \twoheadrightarrow S$ must have $\supp(E) = [s(S),t(T)]_n$ (where we are identifying $E$ with its arc). This shows that there cannot be an indecomposable extension unless $[s(S),t(T)]_n$ does not contain a relation. Thus assume $[s(S),t(T)]_n$ does not contain a relation. It follows that, considered as representations, $S, T$, and any extension between them are supported only on the arrows $\gamma_{s(S)},\ldots,\gamma_{(t(T)-1)_n}$. Thus we can consider $S, T$, and $E$ as representations of the subquiver containing only the arrows $\gamma_{s(S)},\ldots,\gamma_{(t(T)-1)_n}$, which by assumption is of type $A_n$. The result is then clear in this case.
	
	Now suppose that $l(S) + l(T) = n$, so we have both $t(S) = s(T)$ and $s(S) = t(T)$. If the quiver of $\Lambda$ is $A_n$, then the result follows as above. Thus assume the quiver of $\Lambda$ is $\Delta_n$. We now observe that neither $S$ nor $T$ is supported on the arrows $\gamma_{(s(S)-1)_n}$ and $\gamma_{(s(T)-1)_n}$. Moreover, any $E$ with an exact sequence $T \hookrightarrow E \twoheadrightarrow S$ can only be supported on at most one of these arrows. The result then follows from an argument analogous to the first case.
	
	Finally, suppose that $l(S) + l(T) > n$, so without loss of generality we have $t(S) = s(T)$ and $s(S) \neq t(T)$. Again, considered as an arc, any $E$ with an exact sequence $T \hookrightarrow E \twoheadrightarrow S$ clearly has $\supp(E) = [s(S),s(S),t(T)]_n$. Thus, assume $[s(S),s(S),t(T)]_n$ does not contain a relation. It follows that in the universal cover, we can consider $S$, $T$, and $E$ to be representations of the subquiver containing only the arrows $\gamma_{s(S)_0},\ldots,\gamma_{s(S)_1},\ldots,\gamma_{((t(T)-1)_n)_1}$, which by assumption is of type $A_n$. The result is then immediate.
\end{proof}

\begin{lem}\label{lem:monos}
	Let $S,T \in \brick\Lambda$. Then there is a monomorphism $S \hookrightarrow T$ if and only if one of the following holds.
	\begin{enumerate}[label=\upshape(\roman*)]
		\item $s(S) = s(T) <_n t(S) <_n t(T)$ and $t(S)$ is a square.
		\item $s(T) <_n s(S) <_n t(S) = t(T)$ and $s(S)$ is a circle.
		\item $s(T) <_n s(S) <_n t(S) <_n t(T)$, $s(S)$ is a circle, and $t(S)$ is a square.
	\end{enumerate}
	Moreover, in the third case, there is an exact sequence $$S \hookrightarrow M(s(T),(t(S)-s(T))_n) \oplus M(s(S),(t(T)-s(S))_n) \twoheadrightarrow T.$$ In the other cases, every exact sequence $T \hookrightarrow E \twoheadrightarrow S$ is split or has $E$ indecomposable.
\end{lem}

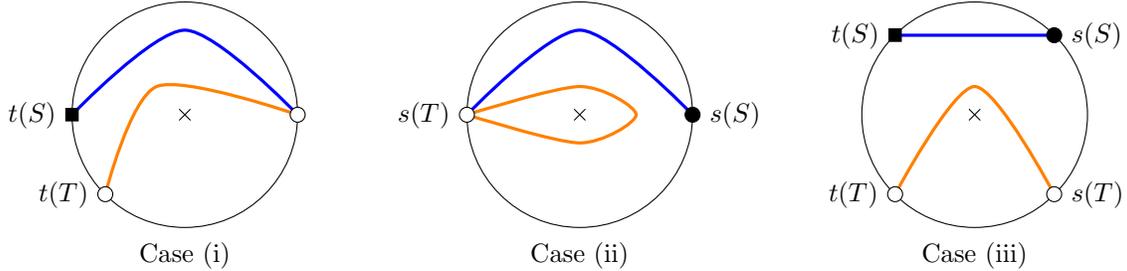
\begin{center}
\begin{tikzpicture}[scale=0.75]
	\draw (0,0) circle[radius=2cm];
	\draw (-.1,-.1) to (.1,.1);
	\draw (-.1,.1) to (.1,-.1);
		\draw[very thick,blue] plot [smooth] coordinates{(2,0)(0,1.5)(-2,0)};
		\draw[very thick,orange,dashed] plot [smooth] coordinates{(2,0)(-0.5,0.5)(-1.41,-1.41)};
		\node at (2,0) [draw, fill=white,circle,scale=0.6]{};
		\node at (-1.41,-1.41) [draw,fill=white,circle,scale=0.6, label=west:$t(T)$]{};
		\node at (-2,0) [draw,fill,rectangle,scale=0.7,label=west:$t(S)$]{};
	\node at (0,-2.5) {Case (i)};
\begin{scope}[shift={(7,0)}]
	\draw (0,0) circle[radius=2cm];
	\draw (-.1,-.1) to (.1,.1);
	\draw (-.1,.1) to (.1,-.1);
		\draw[very thick,blue] plot [smooth] coordinates{(2,0)(0,1.5)(-2,0)};
		\draw[very thick,orange,dashed] plot [smooth] coordinates{(-2,0)(0,-0.5)(1,0)(0,0.5)(-2,0)};
		\node at (-2,0) [draw, fill=white,circle,scale=0.6, label=west:$s(T)$]{};
		\node at (2,0) [draw, fill,circle,scale=0.6, label=east:$s(S)$]{};
		\node at (0,-2.5) {Case (ii)};
\end{scope}
\begin{scope}[shift={(14,0)}]
	\draw (0,0) circle[radius=2cm];
	\draw (-.1,-.1) to (.1,.1);
	\draw (-.1,.1) to (.1,-.1);
		\draw[very thick,blue] plot [smooth] coordinates{(-1.41,1.41)(1.41,1.41)};
		\draw[very thick,orange,dashed] plot [smooth] coordinates{(-1.41,-1.41)(0,0.5)(1.41,-1.41)};
		\node at (1.41,-1.41) [draw, fill=white,circle,scale=0.6, label=east:$s(T)$]{};
		\node at (-1.41,-1.41) [draw,fill=white,circle,scale=0.6, label=west:$t(T)$]{};
		\node at (1.41,1.41) [draw,fill,circle,scale=0.6,label=east:$s(S)$]{};
		\node at (-1.41,1.41) [draw,fill,rectangle,scale=0.7,label=west:$t(S)$]{};
		\node at (0,-2.5) {Case (iii)};
\end{scope}
\end{tikzpicture}
\end{center}

\begin{proof}
	In all three cases, we observe that $S, T$, and any extensions can be considered as representations of the subquiver consisting only of the arrows $\{\gamma_{s(T)},\ldots,\gamma_{((t(T)-1)_n)}\}$, which by assumption is of type $A_n$. The result is then immediate.
\end{proof}

\begin{lem}\label{lem:epis}
	Let $S\ncong T \in \brick\Lambda$. Then there is an epimorphism $S \twoheadrightarrow T$ if and only if one of the following holds.
	\begin{enumerate}[label=\upshape(\roman*)]
		\item $s(S) = s(T) <_n t(S) <_n t(T)$ and $t(T)$ is a circle.
		\item $s(T) <_n s(S) <_n t(S) = t(T)$ and $s(T)$ is a square.
		\item $s(T) <_n s(S) <_n t(S) <_n t(T)$, $s(T)$ is a square, and $t(T)$ is a circle.
	\end{enumerate}
	Moreover, in the third case, there is an exact sequence $$S \hookrightarrow M(s(T),(t(S)-s(T))_n) \oplus M(s(S),(t(T)-s(S))_n) \twoheadrightarrow T.$$ In the other cases, every exact sequence $T \hookrightarrow E \twoheadrightarrow S$ is split or has $E$ indecomposable.
\end{lem}

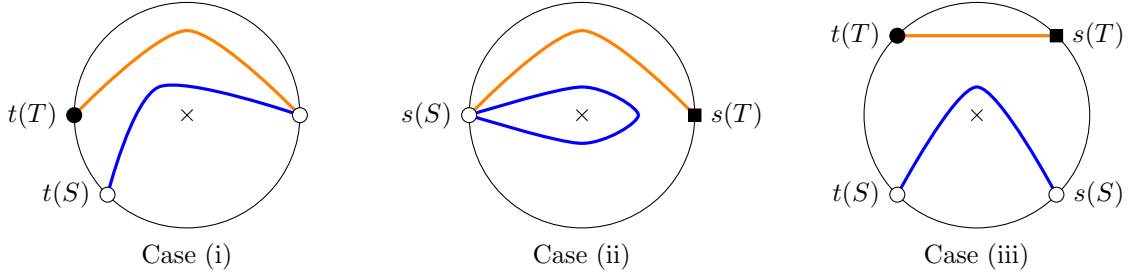
\begin{center}
\begin{tikzpicture}[scale=0.75]
	\draw (0,0) circle[radius=2cm];
	\draw (-.1,-.1) to (.1,.1);
	\draw (-.1,.1) to (.1,-.1);
		\draw[very thick,orange,dashed] plot [smooth] coordinates{(2,0)(0,1.5)(-2,0)};
		\draw[very thick,blue] plot [smooth] coordinates{(2,0)(-0.5,0.5)(-1.41,-1.41)};
		\node at (2,0) [draw, fill=white,circle,scale=0.6]{};
		\node at (-1.41,-1.41) [draw,fill=white,circle,scale=0.6, label=west:$t(S)$]{};
		\node at (-2,0) [draw,fill,circle,scale=0.6,label=west:$t(T)$]{};
	\node at (0,-2.5) {Case (i)};
\begin{scope}[shift={(7,0)}]
	\draw (0,0) circle[radius=2cm];
	\draw (-.1,-.1) to (.1,.1);
	\draw (-.1,.1) to (.1,-.1);
		\draw[very thick,orange,dashed] plot [smooth] coordinates{(2,0)(0,1.5)(-2,0)};
		\draw[very thick,blue] plot [smooth] coordinates{(-2,0)(0,-0.5)(1,0)(0,0.5)(-2,0)};
		\node at (-2,0) [draw, fill=white,circle,scale=0.6, label=west:$s(S)$]{};
		\node at (2,0) [draw, fill,rectangle,scale=0.7, label=east:$s(T)$]{};
		\node at (0,-2.5) {Case (ii)};
\end{scope}
\begin{scope}[shift={(14,0)}]
	\draw (0,0) circle[radius=2cm];
	\draw (-.1,-.1) to (.1,.1);
	\draw (-.1,.1) to (.1,-.1);
		\draw[very thick,orange,dashed] plot [smooth] coordinates{(-1.41,1.41)(1.41,1.41)};
		\draw[very thick,blue] plot [smooth] coordinates{(-1.41,-1.41)(0,0.5)(1.41,-1.41)};
		\node at (1.41,-1.41) [draw, fill=white,circle,scale=0.6, label=east:$s(S)$]{};
		\node at (-1.41,-1.41) [draw,fill=white,circle,scale=0.6, label=west:$t(S)$]{};
		\node at (1.41,1.41) [draw,fill,rectangle,scale=0.7,label=east:$s(T)$]{};
		\node at (-1.41,1.41) [draw,fill,circle,scale=0.6,label=west:$t(T)$]{};
		\node at (0,-2.5) {Case (iii)};
\end{scope}
\end{tikzpicture}
\end{center}

\begin{proof}
	The proof is nearly identical to that of Lemma \ref{lem:monos}.
\end{proof}

\begin{rem}
	It follows from the proofs of Lemma \ref{lem:monos} and Lemma \ref{lem:epis} that if $s(T) <_n s(S) <_n t(S) <_n t(T)$ (so that $S$ and $T$ do not intersect) and $s(S)$ and $t(S)$ are either both squares or both circles that $S$ and $T$ are Ext-orthogonal.
\end{rem}

\begin{lem}\label{lem:notmonoepi}
	Let $S,T \in \brick\Lambda$. Then there is a morphism $S \rightarrow T$ which is neither mono nor epi if and only if one of the following holds.
	\begin{enumerate}[label=\upshape(\roman*)]
		\item $s(S)$ and $t(T)$ are both circles, $s(T) <_n s(S) <_n t(T)$, and $s(S) <_n t(T) <_n t(S)$.
		\item $t(S)$ and $s(T)$ are both squares, $s(S) <_n s(T) <_n t(S)$, and $s(T) <_n t(S) <_n t(T)$.
	\end{enumerate}
	Moreover, in either case, there is an exact sequence $$S \hookrightarrow M(s(S),(t(T)-s(S))_n) \oplus M(s(T),(t(S)-s(T))_n) \hookrightarrow T$$ if and only if $l(S) + l(T) \leq n$ and neither $[s(T),t(S)]_n$ nor $[s(S),t(T)]_n$ contains a relation. Likewise, in case (i), there is an exact sequence $$S \hookrightarrow M(s(S),(t(T)-s(S))_n) \oplus M(s(T),n + (t(S)-s(T))_n) \hookrightarrow T$$ if and only if $l(S) + l(T) > n$ and $[s(T),s(T),t(S)]_n$ does not contain a relation. Finally, in case (ii), there is an exact sequence $$S \hookrightarrow M(s(S),n + (t(T)-s(S))_n) \oplus M(s(T),(t(S)-s(T))_n) \hookrightarrow T$$ if and only if $l(S) + l(T) > n$ and $[s(S),s(S),t(T)]_n$ does not contain a relation. Otherwise, every exact sequence $S \hookrightarrow E \twoheadrightarrow T$ is split or has $E$ indecomposable.
\end{lem}

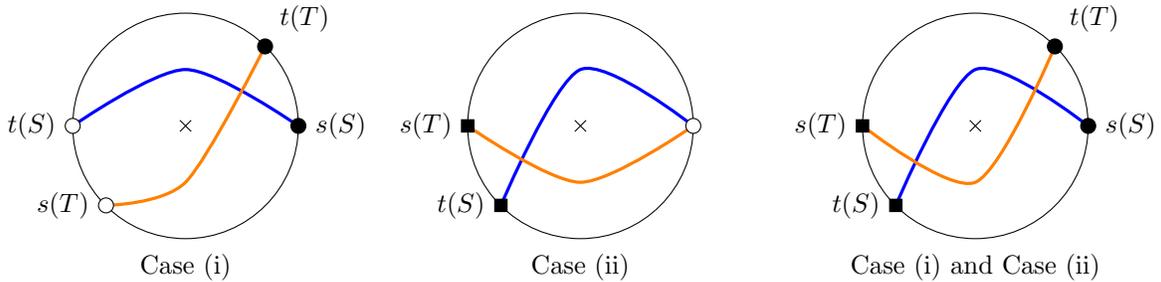
\begin{center}
\begin{tikzpicture}[scale=0.75]
	\draw (0,0) circle[radius=2cm];
	\draw (-.1,-.1) to (.1,.1);
	\draw (-.1,.1) to (.1,-.1);
		\draw[very thick,blue] plot [smooth] coordinates{(2,0)(0,1)(-2,0)};
		\draw[very thick,orange,dashed] plot [smooth] coordinates{(-1.41,-1.41)(0,-1)(1.41,1.41)};
		\node at (1.41,1.41) [draw,fill,circle,scale=0.6,label=north east:$t(T)$]{};
		\node at (-2,0) [draw,fill=white,circle,scale=0.6,label=west:$t(S)$]{};
		\node at (2,0) [draw,fill,circle,scale=0.6,label=east:$s(S)$]{};
		\node at (-1.41,-1.41) [draw,fill=white,circle,scale=0.6,label=west:$s(T)$]{};
		\node at (0,-2.5) {Case (i)};
\begin{scope}[shift={(7,0)}]
	\draw (0,0) circle[radius=2cm];
	\draw (-.1,-.1) to (.1,.1);
	\draw (-.1,.1) to (.1,-.1);
		\draw[very thick,blue] plot [smooth] coordinates{(2,0)(0,1)(-1.41,-1.41)};
		\draw[very thick,orange,dashed] plot [smooth] coordinates{(-2,0)(0,-1)(2,0)};
		\node at (-2,0) [draw,fill,rectangle,scale=0.7,label=west:$s(T)$]{};
		\node at (2,0) [draw,fill=white,circle,scale=0.6]{};
		\node at (-1.41,-1.41) [draw,fill,rectangle,scale=0.7,label=west:$t(S)$]{};
		\node at (0,-2.5) {Case (ii)};
\end{scope}
\begin{scope}[shift={(14,0)}]
	\draw (0,0) circle[radius=2cm];
	\draw (-.1,-.1) to (.1,.1);
	\draw (-.1,.1) to (.1,-.1);
		\draw[very thick,blue] plot [smooth] coordinates{(2,0)(0,1)(-1.41,-1.41)};
		\draw[very thick,orange,dashed] plot [smooth] coordinates{(-2,0)(0,-1)(1.41,1.41)};
		\node at (1.41,1.41) [draw,fill,circle,scale=0.6,label=north east:$t(T)$]{};
		\node at (-2,0) [draw,fill,rectangle,scale=0.7,label=west:$s(T)$]{};
		\node at (2,0) [draw,fill,circle,scale=0.6,label=east:$s(S)$]{};
		\node at (-1.41,-1.41) [draw,fill,rectangle,scale=0.7,label=west:$t(S)$]{};
		\node at (0,-2.5) {Case (i) and Case (ii)};
\end{scope}
\end{tikzpicture}
\end{center}

\begin{proof}
	Suppose there is a morphism $S \rightarrow T$ which is neither mono nor epi. It follows that there exists $R \in \mods\Lambda$ and a chain of morphisms $S \twoheadrightarrow R \hookrightarrow T$. As $S$ and $T$ are bricks, we observe that $R$ cannot be a submodule of $S$ nor a quotient of $T$. It follows that $\supp(S) \nsubseteq \supp(T)$ and vice versa (where $S$ and $T$ are considered as arcs the supports are considered as ordered multisets). This shows that there can only be such a morphism if either both $s(T) <_n s(S) <_n t(T)$ and $s(S) <_n t(T) <_n t(S)$ or both $s(S) <_n s(T) <_n t(S)$ and $s(T) <_n t(S) <_n t(T)$. It can then be verified directly that such a morphism exists only when $s(S)$ and $t(T)$ are both circles or $t(S)$ and $s(T)$ are both squares.
	
	We now consider the results regarding extensions. Suppose first that the arcs corresponding to $S$ and $T$ intersect only once, so case (i) and (ii) cannot occur simultaneously. We consider only case (i), as the proof in case (ii) is nearly identical. Thus we have $s(T) <_n s(S) <_n t(T) <_n t(S)$. We ignore for now that both $t(S)$ and $s(T)$ are circles. As before, we can then consider $S, T$, and any extension as being representations of the subquiver containing only the arrows $\gamma_{s(T)},\ldots,\gamma_{(t(S)-1)_n}$. Now clearly if there exists an exact sequence $S \hookrightarrow E \twoheadrightarrow T$ which is not split, $E$ is supported on every arrow in $\{\gamma_{s(T)},\ldots,\gamma_{(t(S)-1)_n}\}$. Therefore $[t(S),s(T)]_n$ must not contain a relation. The result then follows from the $A_n$ case. In particular, such an extension does not exist unless $t(S)$ and $s(T)$ are both circles.
	
	Now consider the case that the arcs intersect twice, but one of the intersections occurs on the boundary of the disk. Again, case (i) and (ii) cannot occur simultaneously, so we consider only case (i). Thus we have $t(S) = s(T) <_n s(S) <_n t(T)$. If the quiver of $\Lambda$ is $A_n$, the result follows as above. Thus assume the quiver of $\Lambda$ is $\Delta_n$. There are then two generic ways to lift $S$ and $T$ to the universal cover. We denote by $\overline{S}, \overline{T}$ the lifts of $S$ and $T$. 
	
	We can first consider the basepoint as $t(S)$. In this case, write $\supp(S) = \supp(\overline{S}) = [s(S),t(S)]_n = \{s(S)_0,\ldots,t(S)_i\}$, where we see that $i \in \{0,1\}$. We can then consider the lifting of $T$ as having $\supp(\overline{T}) = \{t(S)_i,\ldots,t(T)_j\}$, where again we see $j \in \{1,2\}$. It follows that $\overline{S}, \overline{T}$, and any extension can be considered as representations of the subquiver of the universal cover containing only the arrows $\gamma_{s(S)_0},\ldots,\gamma_{((t(T)-1)_n)_j}$. By similar arguments as before, we observe that there is a non-split exact sequence $\overline{S} \hookrightarrow \overline{E} \twoheadrightarrow \overline{T}$ if and only if $[s(T),s(S)]_n$ does not contain a relation. In this case, we see that the projection of $\overline{E}$ back to the quiver gives an exact sequence $S\hookrightarrow E \twoheadrightarrow T$ with $E \cong M(s(S),(t(T)-s(S))_n) \oplus M(s(T),(t(S)-s(T))_n)$.
	
	To make sure we have accounted for all extensions, we also consider the other generic lifting, with basepoint $s(S)$. In this case, we see that $\overline{S}$ and $\overline{T}$ intersect only at their endpoint, so all additional extensions will be indecomposable as in Lemma \ref{lem:indecomposableExts}.
	
	The last case to consider is when the arcs corresponding to $S$ and $T$ intersect twice in the interior of the disk. This is nearly identical to the previous case, considering again the lifts with basepoints $s(S)$ and $t(S)$.
\end{proof}

\begin{rem}
	It follows from the proof of Lemma \ref{lem:notmonoepi} that if the ordering of $s(S),t(T),s(T),$ and $t(T)$ are given as in case (i) (resp. case (ii)) in the statement of the lemma, but $s(S)$ and $t(T)$ are not both circles (resp. $t(S)$ and $s(T)$ are not both squares), then every exact sequence $T \hookrightarrow E \twoheadrightarrow S$ is split or has $E$ indecomposable.
\end{rem}

We now give two results that will be critical in showing that Nakayama-like algebras have the 2-simple minded pairwise compatibility property.

\begin{prop}\label{prop:extdim1}\
	\begin{enumerate}[label=\upshape(\alph*)]
		\item Let $S, T \in \brick\Lambda$. Then for any nonsplit exact sequence $S \hookrightarrow E \twoheadrightarrow T$, $E$ is given in one of Lemmas \ref{lem:indecomposableExts}, \ref{lem:monos}, \ref{lem:epis}, \ref{lem:notmonoepi}.
		\item Let $S, T \in \brick\Lambda$. Then $\dim\Ext(T,S) \leq 1$.
	\end{enumerate}
\end{prop}

\begin{proof}
	(1) Let $S,T \in \brick\Lambda$. Lemma \ref{lem:indecomposableExts} has already characterized all possible indecomposable extensions between $S$ and $T$. Thus all other extensions occur when $\supp(S) \cap \supp(T) \neq \emptyset$. Thus either $\supp(S) \subset \supp(T)$, in which case we are in the setting of Lemmas \ref{lem:monos} and \ref{lem:epis}, or neither is contained in the other, in which case we are in the setting of Lemma \ref{lem:notmonoepi}.
	
	(2) Suppose $\dim\Ext(T,S) > 1$. If $\supp(S) \subsetneq \supp(T)$ or $\supp(T) \subsetneq \supp(S)$, the result is clear. Thus in order for there to be at least two extensions, we must have that $s(S)$ and $t(T)$ are both circles, $t(S)$ and $s(T)$ are both squares, and neither $[s(S),t(T)]_n$ (resp. $[s(S),s(S),t(T)]_n$) nor $[s(T),t(S)]_n$ (resp. $[s(T),s(T),t(S)]_n$) contains a relation. However, since there exist both marked points which are circles and which are squares, one of these intervals will contains a relation, a contradiction. See the picture below for an illustration.
\end{proof}

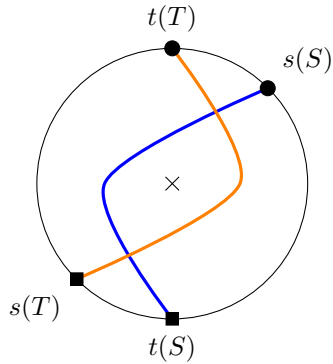
\begin{center}
\begin{tikzpicture}[scale=0.9]
	\draw (0,0) circle[radius=2cm];x
	\draw (-.1,-.1) to (.1,.1);
	\draw (-.1,.1) to (.1,-.1);
		\draw[very thick,blue] plot [smooth] coordinates{(1.41,1.41)(-1,0)(0,-2)};
		\draw[very thick,orange,dashed] plot [smooth] coordinates{(-1.41,-1.41)(1,0)(0,2)};
		\node at (1.41,1.41) [draw,fill,circle,scale=0.6,label=north east:$s(S)$]{};
		\node at (0,-2) [draw,fill,rectangle,scale=0.7,label=below:$t(S)$]{};
		\node at (-1.41,-1.41) [draw,fill,rectangle,scale=0.7,label=south west:$s(T)$]{};
		\node at (0,2) [draw,fill,circle,scale=0.6,label=above:$t(T)$]{};
\end{tikzpicture}
\end{center}

\begin{cor}\label{cor:nakayamaSelfExt}
	Let $\Lambda$ be a Nakayama-like algebra. If there exists $S \in \brick\Lambda$ such that $S$ has a nontrivial self-extension, then $\Lambda$ is a Nakayama algebra.
\end{cor}
\begin{proof}
	Let $S \in \brick\Lambda$ and suppose there is a non-split exact sequence $S \hookrightarrow E\twoheadrightarrow S$. It follows from Lemma \ref{lem:indecomposableExts} that $l(S) = n$ and that $[s(S),s(S),s(S)]_n$ does not contain a relation. In particular, the quiver of $\Lambda$ is $\Delta_n$. Therefore, since $\Lambda$ is Nakayama-like, $I \neq (0)$. As $[i,i]_n \subset [s(S),s(S),s(S)]_n$ for all $1 \leq i \leq n$, this implies that all arrows of the quiver of $\Lambda$ must be oriented in the same direction. We conclude that $\Lambda$ is a Nakayama algebra.
\end{proof}

\begin{cor}\label{cor:nakayamahomdims}
	Let $\Lambda$ be a Nakayama-like algebra. Then for all $S,T\in\brick\Lambda$, we have the following.
		\begin{enumerate} [label=\upshape(\alph*)]
			\item If there exists a morphism $S \rightarrow T$ which is mono or epi, then it is a left minimal $(\Filt T)$-approximation.
			\item If there does not exist a morphism $S\rightarrow T$ which is mono or epi, then there does not exist a left minimal $(\Filt T)$-approximation $S \rightarrow T_S$ which is mono or epi.
		\end{enumerate}
\end{cor}

\begin{proof}
	(a) Such morphisms are completely characterized in Lemmas \ref{lem:monos} and \ref{lem:epis}. We see that in all cases, $\dim\Hom(S,T) = 1$. If $T$ does not have self extensions we are done. Otherwise, $\Lambda$ is Nakayama by Corollary \ref{cor:nakayamaSelfExt}. The result is shown explicitly in this case in \cite[Cor. 3.5]{hanson_tau}.
	
	(b) This follows immediately from Lemma \ref{lem:notmonoepi} because in all three cases, $\supp(S) \nsubseteq \supp(T)$ and vice versa.
\end{proof}

We are now ready to prove that $\Lambda$ has the 2-simple minded pairwise compatibility property.

\begin{thm}\label{thm:cyclicpairwise}
	Let $\Lambda$ be a Nakayama-like algebra and let $\X = \S_p \sqcup \S_n[1]$ be a singly left mutation compatible semibrick pair. Then for every $S \in \S_p$, the semibrick pair $\mu^+_{S,\X}(\X)$ is singly left mutation compatible.
\end{thm}

\begin{proof}
	Let $\X = \S_p \sqcup \S_n[1]$ be a singly left mutation compatible semibrick pair. If $\S_p = \emptyset$, then we are done. Otherwise, let $S \in \S_p$ and let $\mu^+_{S,\X}(\X) =: \X' = \S'_p \sqcup \S'_n[1]$. Assume there exists $X' \in \S'_p$ and $Y' \in \S'_n$ such that there is a map $Y' \rightarrow X'$ which is not mono or epi. By Corollary \ref{cor:nakayamahomdims}, this is equivalent to assuming there exists a left minimal $(\Filt X')$-approximation $Y' \rightarrow X'_{Y'}$ which is not mono or epi.
	
	We first observe that $\X'$ is singly right mutation compatible at $S$ by Proposition \ref{prop:mutation}(a), so $Y' \neq S$. We can thus assume without loss of generality that $\X' = X' \sqcup Y'[1]\sqcup S[1]$. Moreover, by the above lemmas, the arcs corresponding to $X'$ and $Y'$ (which we also refer to as $X'$ and $Y'$ by abuse of notation) intersect (at least once) in the interior of the disk. By Lemma \ref{lem:notmonoepi}, we can thus assume without loss of generality that we have $s(X') <_n s(Y') <_n t(X')$ and $s(Y') <_n t(X') <_n t(Y')$, that $s(Y')$ and $t(X')$ are circles, and that $[s(X'),t(Y')]_n$ (resp. $[s(X'),s(X'),t(Y')]_n$ if $s(X') <_n t(Y') <_n s(Y')$) contains a relation. Moreover, this relation cannot be contained in either interval $[s(X'),t(X'))_n$ or the interval $[s(Y'),t(Y'))$, which implies that the relation contains the interval $[s(Y'),t(X')]_n$. In particular, this means all marked points in the interval $[s(Y'),t(X')]_n$ are circles. The generic diagram is shown below, where it is possible that $s(X') = t(Y')$ or their order is flipped.
	
	
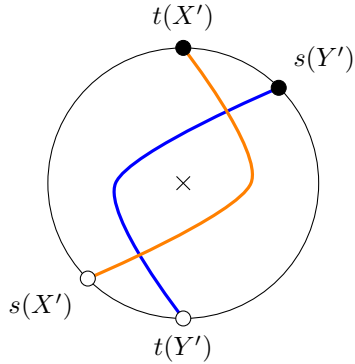
\begin{center}
\begin{tikzpicture}[scale=0.9]
	\draw (0,0) circle[radius=2cm];
	\draw (-.1,-.1) to (.1,.1);
	\draw (-.1,.1) to (.1,-.1);
		\draw[very thick,blue] plot [smooth] coordinates{(1.41,1.41)(-1,0)(0,-2)};
		\draw[very thick,orange,dashed] plot [smooth] coordinates{(-1.41,-1.41)(1,0)(0,2)};
		\node at (1.41,1.41) [draw,fill,circle,scale=0.6,label=north east:$s(Y')$]{};
		\node at (0,-2) [draw, fill=white,circle,scale=0.6,label=below:$t(Y')$]{};
		\node at (-1.41,-1.41) [draw,fill=white,circle,scale=0.6,label=south west:$s(X')$]{};
		\node at (0,2) [draw,fill,circle,scale=0.6,label=above:$t(X')$]{};
\end{tikzpicture}
\end{center}
	
	Claim 1: $Y' \notin \S_n$. Indeed, assume otherwise and recall that $S$ and $Y'$ are Hom orthogonal. Since $\X$ is singly left mutation compatible, it follows that $X' \notin \S_p$. There are then two possibilities.
	
	Suppose first that there exists $X \in \S_n$ such that $X' = \mu^+_{S,\X}(X[1])$, so there is an exact sequence $X \hookrightarrow S \twoheadrightarrow X'$. It follows from Proposition \ref{prop:extdim1} that $l(S) > l(X')$ and either $t(S) = t(X')$ or $s(S) = s(X')$. We observe that if $t(S) = t(X')$ then by Lemma \ref{lem:notmonoepi}, there is a map $Y' \rightarrow S$, a contradiction. Thus we have $s(S) = s(X')$, and hence $t(S) = t(X)$, as shown below where $X$ is the dotted green arc.
	
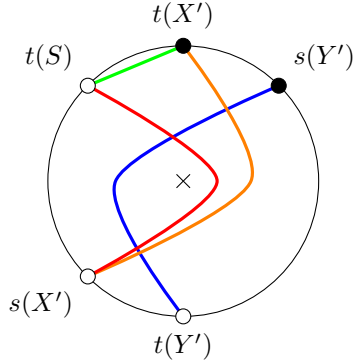
\begin{center}
\begin{tikzpicture}[scale=0.9]
	\draw (0,0) circle[radius=2cm];
	\draw (-.1,-.1) to (.1,.1);
	\draw (-.1,.1) to (.1,-.1);
		\draw[very thick,blue] plot [smooth] coordinates{(1.41,1.41)(-1,0)(0,-2)};
		\draw[very thick,orange] plot [smooth] coordinates{(-1.41,-1.41)(1,0)(0,2)};
		\draw[very thick, red] plot [smooth] coordinates {(-1.41,-1.41)(0.5,0)(-1.41,1.41)};
		\draw[very thick, green,dotted] plot [smooth] coordinates {(0,2)(-1.41,1.41)};
		\node at (1.41,1.41) [draw,fill,circle,scale=0.6,label=north east:$s(Y')$]{};
		\node at (0,-2) [draw, fill=white,circle,scale=0.6,label=below:$t(Y')$]{};
		\node at (-1.41,-1.41) [draw,fill=white,circle,scale=0.6,label=south west:$s(X')$]{};
		\node at (0,2) [draw,fill,circle,scale=0.6,label=above:$t(X')$]{};
		\node at (-1.41,1.41) [draw,fill=white,circle,scale=0.6,label=north west:$t(S)$]{};
\end{tikzpicture}
\end{center}
	
	We observe that $t(X') <_n t(S) <_n t(Y')$ since $S$ is a brick and $[s(X'),t(Y')]_n$ contains a relation if $t(Y') \leq_n s(X') <_n s(Y')$. Now, if $t(S)$ is a circle, then there is a map $Y' \rightarrow S$ by Lemma \ref{lem:notmonoepi}. Otherwise, $t(S)$ is a square and by Lemma \ref{lem:monos} there is a map $X \rightarrow Y'$, a contradiction. This shows that there does not exist $X \in \S_n$ such that $X' = \mu^+_{S,\X}(X[1])$.
	
	It follows that there must exist $X \in \S_p$ such that $X' = \mu^+_{S,\X}(X)$, so there is an exact sequence $S \hookrightarrow X' \twoheadrightarrow X$. Thus $l(S) < l(X')$ and either $s(S) = s(X')$ or $t(S) = t(X')$ by Proposition \ref{prop:extdim1}.
	
	Assume first that $s(S) = s(X')$. Then since $S \hookrightarrow X'$, it follows from Lemma \ref{lem:monos} that $t(S) = s(X)$ is a square. Since every marked point in the interval $[s(Y'),t(X')]_n$ is a circle, this means $t(S) \in (s(X'),s(Y'))_n$, as shown below where $X$ is the dotted green arc.
	
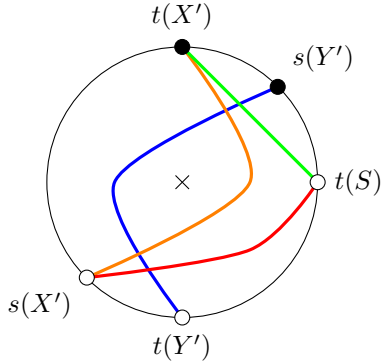
\begin{center}
\begin{tikzpicture}[scale=0.9]
	\draw (0,0) circle[radius=2cm];
	\draw (-.1,-.1) to (.1,.1);
	\draw (-.1,.1) to (.1,-.1);
		\draw[very thick,blue] plot [smooth] coordinates{(1.41,1.41)(-1,0)(0,-2)};
		\draw[very thick,orange] plot [smooth] coordinates{(-1.41,-1.41)(1,0)(0,2)};
		\draw[very thick, red] plot [smooth] coordinates {(-1.41,-1.41)(1,-1)(2,0)};
		\draw[very thick, green,dotted] plot [smooth] coordinates {(0,2)(2,0)};
		\node at (1.41,1.41) [draw,fill,circle,scale=0.6,label=north east:$s(Y')$]{};
		\node at (0,-2) [draw, fill=white,circle,scale=0.6,label=below:$t(Y')$]{};
		\node at (-1.41,-1.41) [draw,fill=white,circle,scale=0.6,label=south west:$s(X')$]{};
		\node at (0,2) [draw,fill,circle,scale=0.6,label=above:$t(X')$]{};
		\node at (2,0) [draw,fill=white,circle,scale=0.6,label=east:$t(S)$]{};
\end{tikzpicture}
\end{center}
	
	We observe that in this case, there is a nonzero map $Y' \rightarrow X$ which is neither mono nor epi by Lemma \ref{lem:notmonoepi}, a contradiction. We conclude that $t(S) = t(X')$. Thus since $\Hom(Y',S) = 0$, we must have that $s(Y') <_n s(S) = t(X) <_n t(X')$. This means $t(X)$ is a circle so as before, there is a map $Y' \rightarrow X$ which is neither mono nor epi, a contradiction. This finishes the proof of Claim 1.
	
	Claim 2: There does not exist $Y \in \S_n$ such that $\mu^+_{S,\X}(Y'[1]) = Y[1]$. Indeed, assume otherwise, so there is an exact sequence $Y' \hookrightarrow Y \twoheadrightarrow S$. Thus $l(Y) > l(Y')$ and either $s(Y) = s(Y')$ or $t(Y) = t(Y')$. There are then three possibilities.
	
	Suppose first that $X' \in \S_p$. If $s(Y) = s(Y')$, then there is a map $Y \rightarrow X'$ which is neither mono nor epi, a contradiction. Thus by Proposition \ref{prop:extdim1}, we have $t(Y) = t(Y')$. As before, we can conclude that $s(X') <_n s(Y) <_n s(Y')$, as pictured below, where $S$ is the dotted green arc.
	
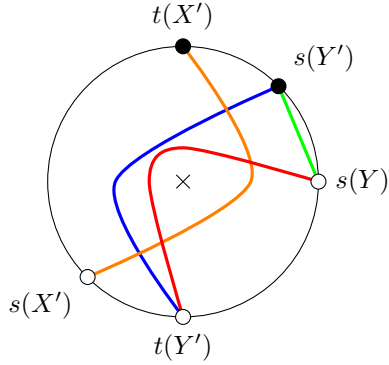
\begin{center}
\begin{tikzpicture}[scale=0.9]
	\draw (0,0) circle[radius=2cm];
	\draw (-.1,-.1) to (.1,.1);
	\draw (-.1,.1) to (.1,-.1);
		\draw[very thick,blue] plot [smooth] coordinates{(1.41,1.41)(-1,0)(0,-2)};
		\draw[very thick,orange] plot [smooth] coordinates{(-1.41,-1.41)(1,0)(0,2)};
		\draw[very thick, red] plot [smooth] coordinates {(2,0)(0,0.5)(-0.5,0)(0,-2)};
		\draw[very thick, green,dotted] plot [smooth] coordinates {(1.41,1.41)(2,0)};
		\node at (1.41,1.41) [draw,fill,circle,scale=0.6,label=north east:$s(Y')$]{};
		\node at (0,-2) [draw, fill=white,circle,scale=0.6,label=below:$t(Y')$]{};
		\node at (-1.41,-1.41) [draw,fill=white,circle,scale=0.6,label=south west:$s(X')$]{};
		\node at (0,2) [draw,fill,circle,scale=0.6,label=above:$t(X')$]{};
		\node at (2,0) [draw,fill=white,circle,scale=0.6,label=east:$s(Y)$]{};
\end{tikzpicture}
\end{center}

	As before if $s(Y)$ is a circle, then there is a map $Y \rightarrow X'$  which is neither mono nor epi. Likewise, if $s(Y)$ is a square, then Lemma \ref{lem:epis} implies that there is a map $X' \rightarrow S$, a contradiction. We conclude that $X' \notin \S_p$.
	
	Now suppose there exists $X \in \S_n$ such that $X = \mu^+_{S,\X}(X[1])$, so there is an exact sequence $X \hookrightarrow S \twoheadrightarrow X'$. Thus as before, $l(S) > l(X')$ and either $s(S) = s(X')$ or $t(S) = t(X')$. If $t(S) = t(X')$, then as before, there is a map $Y'\rightarrow S$, a contradiction. Thus we have $s(S) = s(X')$. Since there is an exact sequence $Y' \hookrightarrow Y \twoheadrightarrow S$, we also know $S$ shares an endpoint with $Y'$. If $s(S) = t(Y')$, then either $t(S) = t(X')$ or $t(Y') <_n t(X') <_n t(S)$ (the latter is only possible if $s(X') = t(Y')$). In either case, we see that $l(Y) = l(Y') + l(S) > n$, contradicting the fact that $Y$ is a brick. Thus we have $t(S) = s(Y')$ , contradicting the fact that $l(S) > l(X')$. This shows there does not exist $X \in \S_n$ such that $X' = \mu_{S,\X}(X[1])$.
	
	The last possibility is that there exists $X \in \S_p$ such that $X' = \mu^+_{S,\X}(X)$, so as before there is an exact sequence $S \hookrightarrow X' \twoheadrightarrow X$. This means $l(S) < l(X')$ and either $s(S) = s(X')$ or $t(S) = t(X')$. Recall that as before $S$ must share an endpoint with $Y'$ as well.
	
	First, suppose $t(Y') = s(X') = s(S)$. Since there is a map $S \rightarrow X'$, Lemma \ref{lem:monos} implies $t(S)$ is a square (and hence is in the interval $(s(S), s(Y'))_n$), as is shown below where $X$ is the dashed red arc and $Y$ is the dotted black arc.
	
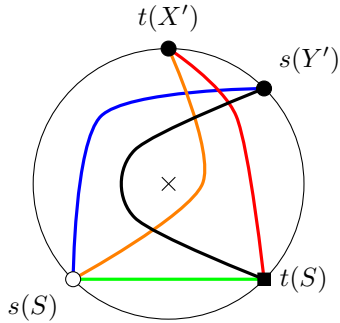
\begin{center}
\begin{tikzpicture}[scale=0.9]
	\draw (0,0) circle[radius=2cm];
	\draw (-.1,-.1) to (.1,.1);
	\draw (-.1,.1) to (.1,-.1);
		\draw[very thick,blue] plot [smooth] coordinates{(1.41,1.41)(-1,1)(-1.41,-1.41)};
		\draw[very thick,orange] plot [smooth] coordinates{(-1.41,-1.41)(0.5,0)(0,2)};
		\draw[very thick, red,dashed] plot [smooth] coordinates {(1.41,-1.41)(1,1)(0,2)};
		\draw[very thick, green] plot [smooth] coordinates {(-1.41,-1.41)(1.41,-1.41)};
		\draw[very thick,dotted] plot [smooth] coordinates {(1.41,1.41)(-.5,.5)(-.5,-.5)(1.41,-1.41)};
		\node at (1.41,1.41) [draw,fill,circle,scale=0.6,label=north east:$s(Y')$]{};
		\node at (-1.41,-1.41) [draw,fill=white,circle,scale=0.6,label=south west:$s(S)$]{};
		\node at (0,2) [draw,fill,circle,scale=0.6,label=above:$t(X')$]{};
		\node at (1.41,-1.41) [draw,fill,rectangle,scale=0.7,label=east:$t(S)$]{};
\end{tikzpicture}
\end{center}

	We observe that by Lemma \ref{lem:notmonoepi} there is a morphism $Y \rightarrow X$ which is neither mono nor epi, a contradiction. If $s(S) = s(X')$ and $t(S) = s(Y')$, then the picture is as below, where $Y$ is the dotted green arc and $S$ is the dashed red arc.
	
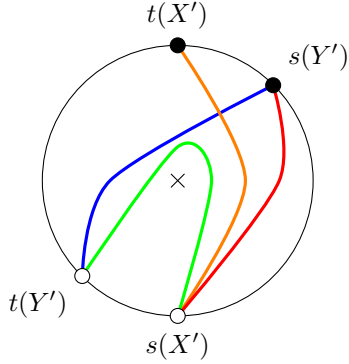
\begin{center}
\begin{tikzpicture}[scale=0.9]
	\draw (0,0) circle[radius=2cm];
	\draw (-.1,-.1) to (.1,.1);
	\draw (-.1,.1) to (.1,-.1);
		\draw[very thick,blue] plot [smooth] coordinates{(1.41,1.41)(-1,0)(-1.41,-1.41)};
		\draw[very thick,orange] plot [smooth] coordinates{(0,-2)(1,0)(0,2)};
		\draw[very thick, red,dashed] plot [smooth] coordinates {(1.41,1.41)(1.5,0)(0,-2)};
		\draw[very thick, green,dotted] plot [smooth] coordinates {(-1.41,-1.41)(0,0.5)(0.5,0)(0,-2)};
		\node at (1.41,1.41) [draw,fill,circle,scale=0.6,label=north east:$s(Y')$]{};
		\node at (0,-2) [draw, fill=white,circle,scale=0.6,label=below:$s(X')$]{};
		\node at (-1.41,-1.41) [draw,fill=white,circle,scale=0.6,label=south west:$t(Y')$]{};
		\node at (0,2) [draw,fill,circle,scale=0.6,label=above:$t(X')$]{};
\end{tikzpicture}
\end{center}

We now know that $t(X') <_n t(Y') \leq_n s(X') <_n s(Y')$ since $S$ is a brick, but this is impossible since $[s(X'),t(Y')]_n$ contains a relation. It follows that $t(S) = t(X')$ and $s(S) = t(Y')$, and as before there is a map $Y' \rightarrow S$, a contradiction. We conclude that there cannot exist $X \in \S_p$ such that $X' = \mu^+_{S,\X}(X)$. This finishes the proof of Claim 2.
	
	Together, Claim 1 and Claim 2 imply that the semibrick pair $\X'$ must be singly left mutation compatible. By Proposition \ref{prop:pairwise2} and Corollary \ref{cor:nakayamahomdims}, this completes the proof.
\end{proof}

\section{The Classification for Gentle Algebras}\label{sec:gentleclassification}

The goal of this section is to prove the central theorem of this paper (Theorem C in the introduction). In particular, this disproves Conjecture 1.11 from \cite{hanson_tau}.

\begin{thm}\label{thm:gentleclassification}
	Let $\Lambda = KQ/I$ be a $\tau$-tilting finite gentle algebra such that $Q$ contains no loops or 2-cycles. Then $\Lambda$ has the pairwise 2-simple minded compatibility property if and only if every vertex of $Q$ has degree at most 2.
\end{thm}

\begin{proof}
	Let $\Lambda = KQ/I$ be $\tau$-tilting finite gentle algebra such that $Q$ contains no loops or 2-cycles. We can assume without loss of generality that $\Lambda$ is connected.
	
	 If the degree of every vertex of $Q$ is at most 2, then $Q \in \{A_n,\Delta_n\}$ and $\Lambda$ is a Nakayama-like algebra. Therefore, by Theorem \ref{thm:cyclicpairwise}, $\Lambda$ has the pairwise 2-simple minded compatibility property.
	
	Thus suppose $Q$ has a vertex of degree at least 3. Since $\Lambda$ is $\tau$-tilting finite, $Q$ contains no multiple edges. This means $Q$ contains one of the following as a subquiver.
	
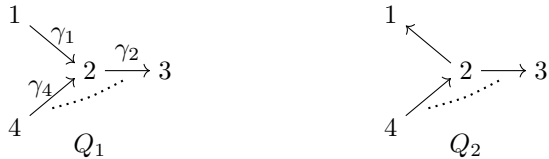
\begin{center}
		\begin{tikzpicture}
			\node at (0,0.75) {1};
			\node at (0,-0.75) {4};
			\node at (1,0) {2};
			\node at (2,0) {3};
			\draw [->] (0.2,0.6)--(0.8,0.1) node[near end,anchor=south] {$\gamma_1$};
			\draw [->] (0.2,-0.6)--(0.8,-0.1) node[near start,anchor=south] {$\gamma_4$};
			\draw [->] (1.2,0)--(1.8,0) node[midway,anchor=south] {$\gamma_2$};
			\draw[thick,dotted] plot [smooth] coordinates{(0.5,-0.5) (1,-0.35) (1.5,-0.1)};
			\node at (1,-1) {$Q_1$};

		\begin{scope}[shift={(5,0)}]
			\node at (0,0.75) {1};
			\node at (0,-0.75) {4};
			\node at (1,0) {2};
			\node at (2,0) {3};
			\draw [<-] (0.2,0.6)--(0.8,0.1);
			\draw [->] (0.2,-0.6)--(0.8,-0.1);
			\draw [->] (1.2,0)--(1.8,0);
			\draw[thick,dotted] plot [smooth] coordinates{(0.5,-0.5) (1,-0.35) (1.5,-0.1)};
			\node at (1,-1) {$Q_2$};
		\end{scope}
		\end{tikzpicture}
	\end{center}
	We will show that if $Q$ contains $Q_1$ as a subquiver, then it does not have the 2-simple minded compatibility property. The argument for $Q_2$ is similar, using $\X = \scriptsize{\begin{matrix}4\\2\\1\end{matrix}} \sqcup 1[1] \sqcup \scriptsize{\begin{matrix}2\\3\end{matrix}}[1]$ if there is no arrow $4 \rightarrow 3$ and $\X = \scriptsize{\begin{matrix}4\\2\\1\end{matrix}} \sqcup 1[1] \sqcup \scriptsize{\begin{matrix}42\\3\end{matrix}}[1]$ if there is an arrow $4\rightarrow 3$.
	
	Suppose $Q$ contains $Q_1$ as a subquiver. If $Q$ does not contain an arrow $4 \rightarrow 3$, let $\X = 1\sqcup \scriptsize{\begin{matrix}4\\2\end{matrix}}\sqcup \scriptsize{\begin{matrix}1\\2\\3\end{matrix}}[1]$. We claim that $\X$ is a singly left mutation compatible semibrick pair and that each subset of $\X$ of size 2 is a mutation compatible semibrick pair.
	
	 Indeed, it is clear that 1 and $\scriptsize{\begin{matrix}4\\2\end{matrix}}$ are Hom orthogonal, that $\scriptsize{\begin{matrix}4\\2\end{matrix}}$ and $\scriptsize{\begin{matrix}1\\2\\3\end{matrix}}$ are Hom orthogonal, and that $\Hom\left(1, \scriptsize{\begin{matrix}1\\2\\3\end{matrix}}\right) = 0$.
	
	To see that $\Ext\left(1,\scriptsize{\begin{matrix}1\\2\\3\end{matrix}}\right) = 0$, it is enough to consider the wide subcategory $\Filt\left(1\sqcup\scriptsize{\begin{matrix}2\\3\end{matrix}}\right)$, which by a result of Jasso (see \cite{jasso_reduction}, \cite[Sec. 4]{demonet_lattice}) is equivalent $\mods\Lambda'$ for some algebra $\Lambda'$ with two simple objects. Let $Q'$ be the quiver of $\Lambda'$. Clearly, $Q'$ contains no loop at $1$, as then $Q$ would as well. Likewise, $Q'$ only contains a single arrow $1 \rightarrow \scriptsize{\begin{matrix}2\\3\end{matrix}}$, otherwise $\Lambda'$ (and hence $\Lambda$) would be $\tau$-tilting infinite. We conclude that $\Ext\left(1,\scriptsize{\begin{matrix}1\\2\\3\end{matrix}}\right) = 0$. Moreover, this also shows that $\dim\Hom\left(\scriptsize{\begin{matrix}1\\2\\3\end{matrix}},1\right) = 1$ and the map $\scriptsize{\begin{matrix}1\\2\\3\end{matrix}} \twoheadrightarrow 1$ is a left minimal $(\Filt 1)$-approximation.
	
	Now suppose $\Ext\left(\scriptsize{\begin{matrix}4\\2\end{matrix}},\scriptsize{\begin{matrix}1\\2\\3\end{matrix}}\right) \neq 0$, so there is a non-split exact sequence
	$$\scriptsize{\begin{matrix}1\\2\\3\end{matrix}} \hookrightarrow E \twoheadrightarrow \scriptsize{\begin{matrix}4\\2\end{matrix}}.$$
	It follows from \cite[Lem. 4.26]{demonet_lattice} that $E$ a brick (and hence is indecomposable). Now, since $\gamma_4\gamma_2$ is a relation, there must be an arrow with source in $\{4,2\}$ and target in $\{1,3\}$. As we have excluded the existence of an arrow $4\rightarrow 3$, the only possibility is that there exists an arrow $4\xrightarrow{\alpha} 1$. However, if this arrow exists then $\alpha\gamma_1$ must be a relation, or else $\Lambda$ would be $\tau$-tilting infinite. We conclude that $\Ext\left(\scriptsize{\begin{matrix}4\\2\end{matrix}},\scriptsize{\begin{matrix}1\\2\\3\end{matrix}}\right) = 0$.
	
	We have shown that $\X$ is a singly left mutation compatible semibrick pair. Moreover, we have that $1 \sqcup \scriptsize{\begin{matrix}4\\2\end{matrix}}$ is mutation compatible since it is a semibrick, $\scriptsize{\begin{matrix}4\\2\end{matrix}} \sqcup \scriptsize{\begin{matrix}1\\2\\3\end{matrix}}[1]$ is mutation compatible since mutating at $\scriptsize{\begin{matrix}4\\2\end{matrix}}$ yields $\scriptsize{\begin{matrix}4\\2\end{matrix}}[1] \sqcup \scriptsize{\begin{matrix}1\\2\\3\end{matrix}}[1]$, and $1 \sqcup \scriptsize{\begin{matrix}1\\2\\3\end{matrix}}[1]$ is mutation compatible since mutating at $1$ yields $1[1] \sqcup \scriptsize{\begin{matrix}2\\3\end{matrix}}[1]$.
	
	Now consider the semibrick pair 
	$$\mu^+_{1,\X}(\X) = \begin{cases}\scriptsize{\begin{matrix}4\\2\end{matrix}} \sqcup 1[1] \sqcup \scriptsize{\begin{matrix}2\\3\end{matrix}}[1] & \textnormal{if there is no arrow }4\rightarrow1\\
		\scriptsize{\begin{matrix}4\\21\end{matrix}} \sqcup 1[1] \sqcup \scriptsize{\begin{matrix}2\\3\end{matrix}}[1] & \textnormal{if there is an arrow }4\rightarrow1
		\end{cases}$$
	We observe that in the first case, the minimal $\left(\Filt \scriptsize{\begin{matrix}4\\2\end{matrix}}\right)$-approximation $\scriptsize{\begin{matrix}2\\3\end{matrix}}\rightarrow \scriptsize{\begin{matrix}4\\2\end{matrix}}$ is neither mono nor epi, and hence $\mu^+_{1,\X}(\X)$ is not singly left mutation compatible. Thus $\Lambda$ does not have the pairwise 2-simple minded compatibility property by Corollary \ref{cor:mutationcompatiblepassdown} The second case is nearly identical.
	
	
	Now suppose that $Q$ does contain an arrow $4\xrightarrow{\alpha} 3$ and let $\X = 1, \scriptsize{\begin{matrix}4\\2\end{matrix}}, \scriptsize{\begin{matrix}1\phantom{1}\\2 4\\3\end{matrix}}[1]$. By arguments analogous to before, we then have that $\X$ is a singly left mutation compatible semibrick pair, each subset of $\X$ of size 2 is a mutation compatible semibrick pair, and $\mu^+_{1,\X}(\X)$ is not singly left mutation compatible.
\end{proof}

\begin{rem}\label{rem:maximal not complete}\
\begin{enumerate}
	\item The counterexamples appearing in the proof of Theorem \ref{thm:gentleclassification} also show that there exist maximal semibrick pairs (in the sense that they are not direct summands of any larger semibrick pairs) which are \emph{not} 2-simple minded collections.
	\item The number of objects in a 2-simple-minded collection is the number of (isoclasses of) simple modules (see \cite[Cor. 5.5]{koenig_silting}). It remains an open question whether there exist ``maximal but not completable'' semibrick pairs with at least this many objects. This will be partially answered in forthcoming work by Emily Barnard and the first author.
\end{enumerate}
\end{rem}

\begin{rem}\label{rem:2cycles}
	It is less clear what the statement of Theorem \ref{thm:gentleclassification} should be if we allow $Q$ to contain loops or 2-cycles. For example, consider the quivers
	\begin{center}
		\begin{tikzpicture}
			\node at (0,0) [anchor=center] {$1 \leftrightarrows 2 \leftrightarrows 3$};
			\node at (0,-0.5) {$Q_1$};
			\node at (3,0) [anchor=center] {$1 \leftrightarrows 2 \leftrightarrows 3 \leftarrow 4$};
			\node at (3,-0.5) {$Q_2$};
		\end{tikzpicture}
	\end{center}
and the ideals $I_1$ and $I_2$ generated by all 2-cycles. We can verify directly that $KQ_1/I_1$ has the 2-simple minded compatibility property. Moreover, we can see that $KQ_2/I_2$ does not using the counterexample $\X = 4 \sqcup \scriptsize{\begin{matrix}2\\3\end{matrix}} \sqcup \scriptsize{\begin{matrix}4\\3\\2\\1\end{matrix}}[1].$ However, both of these algebras are $\tau$-tilting finite and gentle.
\end{rem}


\section{Application: Eilenberg-MacLane Spaces for Picture Groups}\label{sec:application}

We begin this section by recalling the definition of the picture group of $\Lambda$. We then give a brief reminder of the relationship between the picture group and the so called \emph{$\tau$-cluster morphism category} of $\Lambda$. This relationship is of particular interest when $\Lambda$ has the 2-simple minded pairwise compatibility property.


\subsection{Torsion Lattices and Picture Groups}\label{sec:picturegroups}

We first give a brief overview of the lattice structure of the poset $\tors\Lambda$.

\begin{thm}\label{thm:polygons}\
\begin{enumerate}[label=\upshape(\alph*)]
	\item \cite[Prop 1.3]{iyama_lattice} The poset $\tors\Lambda$ is a lattice. That is:
		\begin{itemize}
			\item For every $\T,\T' \in \tors\Lambda$, there exists a unique torsion class $\T \wedge \T'$ such that $\T \wedge \T' \subset \T''$ whenever $\T,\T' \subset \T''$. The torsion class $\T \wedge \T'$ is called the \emph{join} of $\T$ and $\T'$.
			\item For every $\T,\T' \in \tors\Lambda$, there exists a unique torsion class $\T \vee \T'$ such that $\T'' \subset \T \vee \T'$ whenever $\T'' \subset \T,\T'$. The torsion class $\T \vee \T'$ is called the \emph{meet} of $\T$ and $\T'$.
			\end{itemize}
	\item \cite[Thm. 2.2.6]{barnard_minimal} Let $\T \rightarrow \T'$ be an arrow in $\Hasse(\tors\Lambda)$ (that is, $\T' \subsetneq \T$ and there does not exist $\T''$ such that $\T' \subsetneq \T'' \subsetneq \T$). Then there exists a unique brick $S \in \T\setminus\T'$ such that $\T = \Filt(\T \cup \{S\})$, called the \emph{brick label} of the arrow.
	\item \cite[Thm. 4.16]{demonet_lattice} The lattice $\tors\Lambda$ is \emph{polygonal}. That is,
		\begin{itemize}
			\item Given arrows $\T \rightarrow \T_1$ and $\T \rightarrow \T_2$ in $\Hasse(\tors\Lambda)$, the interval $[\T_1\wedge \T_2,\T]$ is a polygon (i.e., the interval $(\T_1\wedge \T_2,\T)$ consists of two disjoint, nonempty strings).
			\item Given arrows $\T_1 \rightarrow \T$ and $\T_2 \rightarrow \T$ in $\Hasse(\tors\Lambda)$, the interval $[\T,\T_1\vee \T_2]$ is a polygon (i.e., the interval $(\T,\T_1\vee \T_2)$ consists of two disjoint, nonempty strings).
		\end{itemize}
	\item \cite[Thm 2.12]{asai_semibricks} Let $\T \in \tors\Lambda$. Let $\S_p$ be the set of all bricks labeling arrows of the form $\T \rightarrow \T'$ in $\Hasse(\tors\Lambda)$ and let $\S_n$ be the set of all bricks labeling arrows of the form $\T' \rightarrow \T$ in $\Hasse(\tors\Lambda)$. Then $\S_p \sqcup \S_n[1]$ is a 2-simple minded collection. Moreover, all 2-simple minded collections occur in this way.
	\item \cite[Prop 4.21]{demonet_lattice} Every polygon in $\Hasse(\tors\Lambda)$ is of the form shown below,
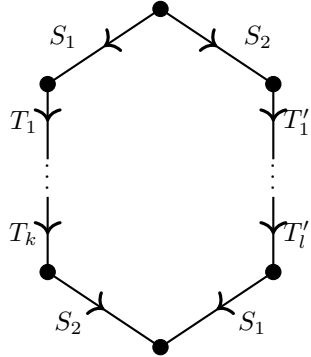
\begin{center}
\begin{tikzpicture}
	\begin{scope}[thick,decoration={
	markings,
	mark=at position 0.5 with {\arrow[scale=1.5]{>}}}
	]
	\draw[postaction={decorate}] (0,-0.5)--(-1.5,-1.5) node [pos=0.65,anchor=south east]{$S_1$};
	\draw[postaction={decorate}] (0,-0.5)--(1.5,-1.5) node [pos=0.65,anchor=south west]{$S_2$};
	\draw[postaction={decorate}] (-1.5,-1.5)--(-1.5,-2.5) node [midway,anchor=east]{$T_1$};
	\draw[postaction={decorate}] (1.5,-1.5)--(1.5,-2.5) node [midway,anchor=west]{$T'_1$};
	\draw[postaction={decorate}] (-1.5,-3)--(-1.5,-4) node [midway,anchor=east]{$T_k$};
	\draw[postaction={decorate}] (1.5,-3)--(1.5,-4) node [midway,anchor=west]{$T'_l$};
	\draw[postaction={decorate}] (-1.5,-4)--(0,-5) node [pos=0.4,anchor=north east]{$S_2$};
	\draw[postaction={decorate}] (1.5,-4)--(0,-5) node [pos=0.4,anchor=north west]{$S_1$};
	\end{scope}
	
	
	\node at (0,-0.5) [draw,fill,circle,scale=0.6]{};
	\node at (0,-5) [draw,fill,circle,scale=0.6]{};
	
	\node at (-1.5,-1.5) [draw,fill,circle,scale=0.6]{};
	\node at (1.5,-1.5) [draw,fill,circle,scale=0.6]{};
	\node at (-1.5,-4) [draw,fill,circle,scale=0.6]{};
	\node at (1.5,-4) [draw,fill,circle,scale=0.6]{};
	
	\node at (-1.5,-2.7)[rotate=90]{\small $\cdots$};
	\node at (1.5,-2.7)[rotate=90]{\small $\cdots$};
\end{tikzpicture}
\end{center}
where $S_1\sqcup S_2 \in \sbrick\Lambda$, $\Filt(S_1\sqcup S_2)$ is a wide subcategory, and $\{S_1,S_2,T_1,\ldots,T_k,T'_1,\ldots,T'_l\}$ is a nonrepeating set of all (isoclasses of) bricks in $\Filt(S_1\sqcup S_2)$\footnote{It is possible that $k=l=0$, that is $S_1$ and $S_2$ are the only bricks in $\Filt(S_1\sqcup S_2)$.}. Moreover, every semibrick of size 2 defines a polygon in $\tors\Lambda$ in this way.
\end{enumerate}
\end{thm}

In the above picture, we refer to $(S_1,T_1,\ldots,T_k,S_2)$ and $(S_2,T'_1,\ldots,T'_l,S_1)$ as the \emph{sides} of the polygon $\P(S_1,S_2)$. This leads to the definition of the picture group of $\Lambda$, first given for representation finite hereditary algebras by the second author, Todorov, and Weyman in \cite{igusa_picture}.

\begin{define}\cite[Def-Thm. 4.3, Prop. 4.4]{hanson_tau}
	The \emph{picture group} of $\Lambda$, denoted $G(\Lambda)$, is the group with the following (equivalent) presentations.
	\begin{enumerate}[label=\upshape(\alph*)]
		\item $G(\Lambda)$ has one generator $X_S$ for every brick $S \in \brick\Lambda$. For every polygon $\P(S_1,S_2)$ in $\tors\Lambda$, we have the polygon relation
		$$X_{S_1}X_{T_1}\cdots X_{T_k}X_{S_2} = X_{S_2}X_{T'_1}\cdots X_{T'_l}X_{S_1}$$
		where $(S_1,T_1,\ldots T_k,S_2)$ and $(S_2,T'_1,\ldots T'_l,S_1)$ the sides of $\P(S_1,S_2)$.
		\item $G(\Lambda)$ has one generator $X_S$ for every brick $S \in \brick\Lambda$ and one generator $g_\T$ for every torsion class $\T \in \tors\Lambda$. There is a relation $$g_0 = e$$ and for every arrow $\T \xrightarrow{S} \T'$ in $\Hasse(\tors\Lambda)$, there is a relation $$g_{\T} = X_S g_{\T'}$$
	\end{enumerate}
\end{define}

An interesting problem is to compute the cohomology of picture groups. In many cases, this can be done by computing the cohomology of a topological space which has the picture group as its fundamental group (see \cite[Sec. 3-5]{igusa_picture}, \cite[Sec. 4]{igusa_signed}). This topological space is the classifying space of a small category known as the \emph{$\tau$-cluster morphism category} of $\Lambda$, defined by Buan and Marsh in \cite{buan_category} to generalize a construction of the second author and Todorov in \cite{igusa_signed}. 

In order to make this paper self-contained, we give the definition of the $\tau$-cluster morphism category below. Readers interested in more information about this category and information about its classifying space are referred to \cite{buan_category}, \cite{buan_exceptional}, \cite[Sec. 3]{igusa_picture}, \cite{igusa_signed}, and \cite[Sec. 2]{hanson_tau}.

\begin{define}\label{def:category}\cite[Sec. 1]{buan_category}
	The $\tau$-cluster morphism category of $\Lambda$, denoted $\W(\Lambda)$ is defined as follows:
	\begin{enumerate}
		\item The objects of $\W(\Lambda)$ are the wide subcategories of $\mods\Lambda$; i.e., $\mathcal{O}b(\W(\Lambda)) = \wide\Lambda$.
		\item Let $W, W' \in \mods\Lambda$. Then
		$$\Hom_{\W(\Lambda)}(W,W') := \{M\sqcup P[1] \in \str W| (M\sqcup P)^\perp \cap \lperp(\tau M) = W'\}.$$
		\item Let $M_1\sqcup P_1[1] \in \Hom_{\W}(W,W')$ and $M_2\sqcup P_2[1] \in \Hom_{\W}(W',W'')$. Then $(M_2\sqcup P_2[1])\circ (M_1\sqcup P_1[1]) = N\sqcup Q[1]$ is the unique element of $\str W$ satisfying:
		\begin{enumerate}
			\item $M_1 \sqcup P_1[1]$ is a direct summand of $N\sqcup Q[1]$.
			\item $W'' = (N\sqcup Q)^\perp \cap \lperp(\tau N)$.
			\item $\Fac M_2 = \Fac N \cap W'$, where $\Fac M_2$ is computed in the category $W'$ and $\Fac N$ is computed in the category $W$.
		\end{enumerate}
	\end{enumerate}
\end{define}

We denote by $\B\W(\Lambda)$ the classifying space of the $\tau$-cluster morphism category of $\Lambda$. This is a simplial complex whose vertices correspond to objects of the category, whose 1-simplices correspond to non-identity morphisms, and whose $k$-simplices correspond to sequences of $k$ composable morphisms. A more formal definition can be found in \cite[Sec. 3.4]{igusa_signed}. We recall that a topological space $X$ is a called an \emph{Eilenberg-MacLane space}, or $K(\pi,1)$, for the group $\pi$ if the cohomology (with arbitrary coefficients) of $X$ is isomorphic to the cohomology of $\pi$. We use only the following facts about $\B\W(\Lambda)$.

\begin{thm}\label{thm:whenKpi1}\
	\begin{enumerate}[label=\upshape(\alph*)]
		\item \cite[Thm. 4.8]{hanson_tau} The fundamental group of $\B\W(\Lambda)$ is isomorphic to the picture group of $\Lambda$. That is, $\pi_1(\B\W(\Lambda)) \cong G(\Lambda)$.
		\item \cite[Prop 3.4, Prop 3.7]{igusa_category} \cite[Thm 2.12, Lem. 2.14]{hanson_tau} Suppose there is a faithful group functor $\W(\Lambda) \rightarrow G$ for some group $G$, considered as a groupoid with one object. If $\Lambda$ has the pairwise 2-simple minded compatibility property, then $\B\W(\Lambda)$ is a $K(G(\Lambda),1)$.
	\end{enumerate}
\end{thm}


\subsection{Faithful Group Functors}\label{sec:faithfulfunctors}

The aim of this section is to construct a faithful group functor $\W(\Lambda) \rightarrow G(\Lambda)$ in as general of a setting as possible. In particular, this will include the case that $\Lambda$ is gentle or Nakayama-like. We do so by showing that such a $\Lambda$ satisfies the hypotheses of the following theorem.

\begin{thm}\cite[Thm. 4.13]{hanson_tau}\label{thm:faithfulFunctor}
	Let $\Lambda$ be a $\tau$-tilting finite algebra such that
		\begin{enumerate}[label=\upshape(\alph*)]
			\item For $S \neq S' \in \brick\Lambda$, we have $e \neq X_S\neq X_{S'} \in G(\Lambda)$.
			\item For $\T \neq \T' \in \tors\Lambda$, we have $g_{\T} \neq g_{\T'} \in G(\Lambda)$.
		\end{enumerate}
	Then there is a faithful functor $F: \W(\Lambda) \rightarrow G(\Lambda)$, where $G(\Lambda)$ is considered as a groupoid with one object.
\end{thm}

Our technique for verifying the hypotheses of Theorem \ref{thm:faithfulFunctor} is similar to that in \cite[Sec. 4.3]{hanson_tau} used for Nakayama algebras. That is, we proceed by mapping $G(\Lambda)$ to another group where we know the generators are nontrivial. As the \emph{brick algebra} defined in \cite{hanson_tau} is in general not associative, we instead map $G(\Lambda)$ to the group of units of the \emph{power series 0-Hall algebra} of $\Lambda$, defined below in Definition \ref{def:hallalg}. We begin by recalling the definition of the Hall polynomials of $\Lambda$.

\begin{defthm}\cite[Thm. 1]{ringel_hall}
	Write $\Lambda = KQ/I$ where $Q$ is a quiver (or more generally, a species) and $I$ is an admissible ideal. We identify $\mods\Lambda$ with $\mods K'Q/I$ for any field $K'$ by identifying the Auslander-Reiten quivers of $\Lambda$ and $K'Q/I$. For $M, N, E\in \mods\Lambda$, let $\psi_{M,N}^E(q)$ be the number of submodules $N' \subset E$ such that $N' \cong N$ and $E/N' \cong M$ when $K' = \mathbb{F}_q$, the finite field with $q$ elements. Then the function $\psi_{M,N}^E$ is polynomial in $q$. The polynomials $\psi_{M,N}^E$ are called the \emph{Hall polynomials} of $\Lambda$.
\end{defthm}

We are now ready to define our algebra.

\begin{define}\label{def:hallalg}
	The \emph{power series 0-Hall algebra} of $\Lambda$, denoted $\H(\Lambda)$, is defined as follows.
	\begin{itemize}
	\item The elements of $\H(\Lambda)$ are formal power series over $\Z$ of isoclasses of $\mods\Lambda$. That is, elements are of the form $\displaystyle \sum_{[M] \in [\mods\Lambda]}c_{[M]}\cdot [M]$, where $[\mods\Lambda]$ is a skeleton of $\mods\Lambda$ and the coefficients $c_{[M]}$ are integers.
	\item Multiplication is defined by $\displaystyle [M] * [N] = \sum_{[E] \in [\mods\Lambda]} \psi_{M,N}^E(0) \cdot [E]$.
	\end{itemize}
	This is an associative algebra with multiplicative identity $[0]$.
\end{define}

In order to construct a morphism $G(\Lambda) \rightarrow \H^*(\Lambda)$, we will need the following.

\begin{lem}\label{lem:hallresults}
	Let $S, T \in \brick\Lambda$.
	\begin{enumerate}[label=\upshape(\alph*)]
		\item If $l(S) + l(T) \neq l(E)$, then the coefficient of $[E]$ in $[S]*[T]$ is 0.
		\item If $\Hom(T,S) = 0$, then the coefficient of $[E]$ in $[S]*[T]$ is 1 if and only if there is an exact sequence $T \hookrightarrow E \twoheadrightarrow S$.
		\item If $\Hom(T,S) \neq 0$, then the coefficient of $[S\oplus T]$ in $[S]*[T]$ is 0.
	\end{enumerate}
\end{lem}

\begin{proof}
	(a) This is clear since $E$ cannot possibly contain a submodule $T'$ such that $T' \cong T$ and $E/T' \cong S$ unless this equality holds.
	
	(b) Suppose there exists an exact sequence $T \hookrightarrow E \twoheadrightarrow S$. Now let $T' \subset E$ such that $T' \cong T$. As $\Hom(T', S) = 0$, it follows that $T' \subset T$. Since $T$ is a brick, this implies $T' = T$. We conclude that $\psi_{S,T}^E = 1$ and hence the coefficient of $[E]$ in $[S]*[T]$ is 1. The converse follows directly from the definition of the Hall polynomials.
	
	(c) Choose an exact sequence $T\xrightarrow{f}S\oplus T \twoheadrightarrow S$. Since $S$ is indecomposable, it follows that $f_2:= pr_2\circ f$ is a nonzero morphism $T \rightarrow T$. It is clear that for every $g \in \Hom(T,S)$, the morphism $T\rightarrow S\oplus T$ given by $x \mapsto (g(x),x)$ defines a distinct submodule $T' \subset S\oplus T$ with $T' \cong T$ and $(S\oplus T)/T' \cong S$. We claim these are the only such submodules. Indeed, let $g \in \Hom(T,S)$ and $0 \neq h \in \End(T)$. Since $T$ is a brick, $h$ is invertible. Thus the image of the morphism $T \rightarrow S\oplus T$ given by $x \mapsto (g(x),h(x))$ is the same as that given by $x \mapsto (g(h\inv(x)),x)$. This shows that $\psi_{S,T}^{S\oplus T}(q) = q^{\dim\Hom(T,S)} \not\equiv 1$. We conclude that the coefficient of $[S\oplus T]$ in $[S]*[T]$ is 0.
\end{proof}

We now propose our group homomorphism.

\begin{thm}\label{thm:knownpolys}
	Assume that for all $S\sqcup T \in \sbrick\Lambda$ we have one of the following.
	\begin{itemize}
		\item $S$ and $T$ are both $K$-stones. That is, $\Ext(S,S) = 0 = \Ext(T,T)$ and $\End(S) \cong K \cong \End(T)$.
		\item $S$ and $T$ are Ext orthogonal.
	\end{itemize}
	Then there is a group homomorphism $\phi_\Lambda: G(\Lambda) \hookrightarrow \H^*(\Lambda)$ given by $\phi_\Lambda(X_S) = (1-[S])\inv$.
\end{thm}

\begin{proof}
	We need only show the proposed map preserves the relations of $G(\Lambda)$. Let $S_1\sqcup S_2 \in \sbrick\Lambda$ and let $\P(S_1\sqcup S_2)$ be the corresponding polygon in $\tors\Lambda$. If $S_1$ and $S_2$ are both $K$-stones, then by \cite[Prop 4.33]{demonet_lattice}, there are only four possibilities for the polygon $\P(S_1,S_2)$.
	\begin{enumerate}
		\item $S_1$ and $S_2$ are Ext orthogonal and the sides of $\P(S_1,S_2)$ are $(S_1,S_2)$ and $(S_2,S_1)$. 
		\item $\Ext(S_1,S_2) = 0 , \Ext(S_2,S_1) \neq 0$, and the sides of $\P(S_1,S_2)$ are $(S_1,T,S_2)$ and $(S_2,S_1)$ where $S_1 \hookrightarrow T \twoheadrightarrow S_2$ is exact.
		\item $\Ext(S_1,S_2) \neq 0 , \Ext(S_2,S_1) = 0$, and the sides of $\P(S_1,S_2)$ are $(S_1,S_2)$ and $(S_2,T,S_1)$ where $S_2 \hookrightarrow T \twoheadrightarrow S_1$ is exact.
		\item $\Ext(S_1,S_2) \neq 0 ,  \Ext(S_2,S_1) \neq 0$, and the sides of $\P(S_1,S_2)$ are $(S_1,T,S_2)$ and $(S_2,R,S_1)$ where $S_1 \hookrightarrow T \twoheadrightarrow S_1$ and $S_2\hookrightarrow R \twoheadrightarrow S_1$ are exact. 	
		\end{enumerate}
	It can be verified in all four cases that the morphism $\phi_\Lambda$ preserves the relation corresponding to $\P(S_1,S_2)$ using Lemma \ref{lem:hallresults}. For example, in case (2), we have
	\begin{eqnarray*}
		\phi_\Lambda(X_{S_2}\inv X_T\inv X_{S_1}\inv) &=& (1-[S_2])*(1-[T])*(1-[S_1])\\
			&=& 1 - [S_1] - [S_2] - [T] + [S_2]*[S_1]\\
			&=& 1 - [S_1] - [S_2] + [S_1\oplus S_2]\\
			&=& 1 - [S_1] - [S_2] + [S_1]*[S_2]\\
			&=& (1-[S_1])*(1-[S_2])\\
			&=& \phi_\Lambda(X_{S_1}\inv X_{S_2}\inv)
		\end{eqnarray*}
		
	Likewise, if $S_1$ and $S_2$ are Ext orthogonal, then $\brick(\Filt(S_1\sqcup S_2)) = \{S_1,S_2\}$ and $\P(S_1,S_2)$ is as in Case (1) above.
	\end{proof}

The assumption that all pairs of Hom-orthogonal bricks are either $K$-stones or are Ext orthogonal is made since the possible polygons which can occur in $\tors\Lambda$ are well understood in these cases. Generalizing this proof would require a classification of all 2-vertex $\tau$-tilting finite algebras, which does not currently exist. We have, however, verified that the theorem still holds for the polygons corresponding to the Dynkin diagrams $B_2 \cong C_2$ and $G_2$, which gives us hope the following conjecture is true.

\begin{conj}\label{conj:faithfulfunctor}
	Let $\Lambda$ be an arbitrary $\tau$-tilting finite algebra. Then there is a group homomorphism $\phi_\Lambda: G(\Lambda) \rightarrow \H^*(\Lambda)$ given by $\phi_\Lambda(X_S) = (1-[S])\inv$.
\end{conj}

Nevertheless, we have the following.

\begin{prop}\label{prop:stones}\
	\begin{enumerate}
		\item Let $\Lambda$ be a Nakayama-like algebra. Then $\Lambda$ satisfies the hypotheses of Theorem \ref{thm:knownpolys}, and therefore the map $\phi_\Lambda: G(\Lambda) \rightarrow \H(\Lambda)$ is a group homomorphism.
		\item Let $\Lambda$ be a $\tau$-tilting finite gentle algebra such that the quiver of $\Lambda$ has no loops or 2-cycles. Then $\Lambda$ satisfies the hypotheses of Theorem \ref{thm:knownpolys}, and therefore the map $\phi_\Lambda: G(\Lambda) \rightarrow \H(\Lambda)$ is a group homomorphism.
	\end{enumerate}
\end{prop}
	
\begin{proof}	
	(1) Let $S\sqcup T\in\sbrick\Lambda$. It is clear in this case that the endomorphism rings of $S$ and $T$ must be isomorphic to the field $K$. Thus if neither $S$ nor $T$ have nontrivial self-extensions, we are done. Thus suppose $S \in \brick\Lambda$ has a nontrivial self-extension. This implies that $\Lambda$ is a Nakayama algebra (Corollary \ref{cor:nakayamaSelfExt}). The result then follows immediately from \cite[Lem. 3.8]{hanson_tau}.
	
	(2) Let $S \in \brick\Lambda$ and suppose there is a nonsplit exact sequence $S \hookrightarrow E \twoheadrightarrow S$. We observe that $E$ must be indecomposable since it has length 2 in the wide subcategory $\Filt(S)$. Identifying $S$ with its string, it then follows from \cite[Thm. 8.5]{brustle_combinatorics} and the fact that $Q$ has no loops that there is an arrow $\alpha \in Q_1$ such that $S\alpha^\epsilon S$ is a string for some $\epsilon \in \{\pm 1\}$. It follows that $S(\alpha^\epsilon S)^m$ is a string for all $m > 0$, so $\Lambda$ is representation infinite and hence $\tau$-tilting infinite by Theorem \ref{thm:finitegentle}, a contradiction.
\end{proof}

When the map $\phi_\Lambda: G(\Lambda) \rightarrow \H(\Lambda)$ is a group homomorphism, we can prove the following result about the generators of $G(\Lambda)$.

\begin{lem}\label{lem:faithfulfunctor}
	Suppose the map $\phi_\Lambda$ is a group homomorphism. Then
	\begin{enumerate}[label=\upshape(\alph*)]
		\item For $S \neq S' \in \brick\Lambda$, we have $e \neq X_S \neq X_{S'} \in G(\Lambda)$.
		\item For $\T \neq \T' \in \brick\Lambda$, we have $g_{\T} \neq g_{\T'} \in G(\Lambda)$.
	\end{enumerate}
	In particular, there exists a faithful group functor $\W(\Lambda) \rightarrow G(\Lambda)$ as claimed in Theorem \ref{thm:faithfulFunctor}.
\end{lem}

\begin{proof}
	The proof is similar to that of \cite[Lem. 4.12]{hanson_tau}.
	
	(a) Let $S\neq S' \in \brick\Lambda$. Then it is clear that $\phi_\Lambda(X_S\inv) = 1-[S] \neq 1-[S'] = \phi_\Lambda(X_{S'}\inv)$ and that neither of these are equal to the identity.
	
	(b) Let $\T, \T' \in \tors\Lambda$ and let $(S_1,\ldots,S_k),(S'_1,\ldots,S'_l)$ label minimal length paths $\T \rightarrow \T\wedge\T',\T' \rightarrow \T\wedge \T'$ in $\Hasse(\tors\Lambda)$. Assume that $X_{S_1}\cdots X_{S_k} = X_{S'_1}\cdots X_{S'_l} \in G(\Lambda)$ (and hence $g_\T = g_{\T'}$). If $\max(k,l) > 0$, let $S \in \{S_1,\ldots,S_k,S'_1,\ldots,S'_l\}$ be of minimal length. By applying the morphism $\phi_\Lambda$, we then have
	$$(1-[S_k])*\cdots*(1-[S_1]) = (1-[S'_l])*\cdots*(1-[S'_1]).$$
	By Lemma \ref{lem:hallresults}, this implies that $S$ occurs in both $(S_1,\ldots,S_k)$ and $(S'_1,\ldots,S'_l)$. Thus $\T \wedge \T' \subsetneq \Filt\Fac(\brick(\T\wedge\T')\sqcup S) \subset \T,\T'$, a contradiction. We conclude that $\max(k,l) = 0$ and hence $\T = \T\wedge\T' = \T'$.
\end{proof}

From this along with Theorems \ref{thm:cyclichaveproperty}, \ref{thm:gentleclassification}, \ref{thm:whenKpi1}(b), we conclude our final main result (Theorem D in the introduction).

\begin{cor}\label{cor:kpi1}\
	\begin{enumerate}[label=\upshape(\alph*)]
		\item Let $\Lambda$ be a Nakayama-like algebra. Then the classifying space of the $\tau$-cluster morphism category of $\Lambda$ is a $K(G(\Lambda),1)$.
		\item Let $\Lambda$ be a $\tau$-tilting finite gentle algebra whose quiver contains no loops or 2-cycles. Then the classifying space of the $\tau$-cluster morphism category of $\Lambda$ is a $K(G(\Lambda), 1)$ if the degree of every vertex of $Q$ is at most 2.
	\end{enumerate}
\end{cor}


\section*{Future Work}

The fact that there are $\tau$-tilting finite algebras which do not have the pairwise 2-simple minded compatibility property opens up several interesting questions for research. First and foremost, we would like to better understand what it is that causes an algebra to fail to have this property. For example, consider the algebra $\Lambda = KQ/I$ which is cluster tilted of type $A_4$, shown below.
	
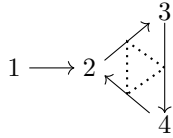
\begin{center}
	\begin{tikzpicture}
		\node at (0,0) {1};
		\node at (2,-0.75) {4};
		\node at (1,0) {2};
		\node at (2,0.75) {3};
		\draw [->] (0.2,0)--(0.8,0) node[midway,anchor=south] {};
		\draw [->] (1.8,-0.6)--(1.2,-0.1) node[midway,anchor=north] {};
		\draw [->] (1.2,0.1)--(1.8,0.6) node[near start,anchor=south] {};
		\draw [->] (2,0.6)--(2,-0.6) node[midway,anchor=west] {};
		\draw[thick,dotted] (1.5,-0.35)--(1.5,0.35)--(2,0)--cycle;
	\end{tikzpicture}
	\end{center}
The only counterexample to the 2-simple minded compatibility property we could find for this algebra is the collection $\X = 1, \scriptsize{\begin{matrix}4\\2\end{matrix}}, \scriptsize{\begin{matrix}1\\2\\3\end{matrix}}[1]$. This seems to indicate the algebra `almost' has the pairwise 2-simple minded compatibility property and hence should either be a $K(G(\Lambda),1)$ or `almost' be a $K(G(\Lambda),1)$ in some imprecise sense. More precisely, we are interested in the following.

\begin{ques}\
	\begin{enumerate}
		\item Is $\B\W(\Lambda)$ a $K(G(\Lambda),1)$? If not, can we obtain from $\B\W(\Lambda)$ (by adding/deleting/pasting cells, etc.) a cube complex that is a $K(G(\Lambda),1)$?
		\item How does the cohomology of $\B\W(\Lambda)$ compare to that of $G(\Lambda)$? Is there a way to compute the cohomology of $G(\Lambda)$ without first finding a $K(G(\Lambda),1)$?
	\end{enumerate}
\end{ques}

We also plan to return to algebras of the form $KQ_{i,j}/I$, as defined in Remark \ref{rem:othercyclic}. In particular, we are interested in the following.

\begin{ques}
	If $KQ_{i,j}/I$ is $\tau$-tilting finite, does it have the 2-simple minded compatibility property?
\end{ques}

A positive answer to this question would show that all algebras of the form $K\Delta_n/I$ which are $\tau$-tilting finite have the 2-simple minded compatibility property. In particular, this would show that in order for an arbitrary $\tau$-tilting finite algebra to fail to have the property, its quiver would need to contain a vertex of degree at least 3.

Along the same lines, we wish to further study ($\tau$-tilting finite) gentle algebras whose quivers contain loops and 2-cycles. In particular, we wish to answer the following.

\begin{ques}
	Let $\Lambda = KQ/I$ be an arbitrary $\tau$-tilting finite gentle algebra. When does $\Lambda$ have the 2-simple minded compatibility property?
\end{ques}

Remark \ref{rem:2cycles} shows that the answer to this question is not that $\Lambda$ has the property if and only if every vertex of $Q$ has degree at most 2. It also shows we cannot replace the requirement that every vertex has degree at most 2 with the requirement that every vertex is connected by an arrow to at most 2 other vertices.

Finally, we wish to resolve Conjecture \ref{conj:faithfulfunctor}. A proof of this conjecture either requires finding a classification of all 2-vertex $\tau$-tilting finite algebras or finding a new technique that does not depend on complete knowledge of which polygons can occur in $\tors\Lambda$.

\section*{Acknowledgements}
The first author is thankful to Emily Barnard, Corey Bregman, and Job Rock for meaningful conversations and suggestions. Both authors would like to thank Gordana Todorov for support and suggestions. The authors also thank an anonymous referee for thoughtful and thorough comments and suggestions.

\bibliographystyle{amsalpha}

\bibliography{../../bibliography-1}

\newcommand{\etalchar}[1]{$^{#1}$}
\providecommand{\bysame}{\leavevmode\hbox to3em{\hrulefill}\thinspace}
\providecommand{\MR}{\relax\ifhmode\unskip\space\fi MR }
\providecommand{\MRhref}[2]{%
  \href{http://www.ams.org/mathscinet-getitem?mr=#1}{#2}
}
\providecommand{\href}[2]{#2}
\begin{thebibliography}{MVdlP83}

\bibitem[Ada16]{adachi_classification}
Takahide Adachi, \emph{The classification of $\tau$-tilting modules over
  {Nakayma} algebras}, J. Algebra \textbf{452} (2016), 227--262.

\bibitem[Aih13]{aihara_tilting}
Takuma Aihara, \emph{Tilting-connected symmetric algebras}, Algebr. Represent.
  Theory \textbf{16} (2013), no.~3, 873--894.

\bibitem[AIR14]{adachi_tilting}
Takahide Adachi, Osamu Iyama, and Idun Reiten, \emph{$\tau$-tilting theory},
  Compos. Math. \textbf{150} (2014), no.~3, 415--452.

\bibitem[APS]{amiot_complete}
Claire Amiot, Pierre-Guy Plamondon, and Sibylle Schroll, \emph{A complete
  derived invariant for gentle algebras via winding numbers and {Arf}
  invariants}, arXiv:1904.02555.

\bibitem[Asa20]{asai_semibricks}
Sota Asai, \emph{Semibricks}, Int. Math. Res. Not. IMRN \textbf{2020} (2020),
  no.~16, 4993--5054.

\bibitem[BCZ19]{barnard_minimal}
Emily Barnard, Andrew~T. Carroll, and Shijie Zhu, \emph{Minimal inclusions of
  torsion classes}, Algebraic Combin. \textbf{2} (2019), no.~5, 879--901.

\bibitem[BDM{\etalchar{+}}19]{brustle_combinatorics}
Thomas Br\"ustle, Giullaume Douville, Kaveh Mousavand, Hugh Thomas, and Emine
  Yildirim, \emph{On the combinatorics of gentle algebras}, Canad. J. Math.
  (2019).

\bibitem[BM]{buan_exceptional}
Aslak~Bakke Buan and Bethany~R. Marsh, \emph{$\tau$-exceptional sequences},
  arXiv:1802.01169.

\bibitem[BM19]{buan_category}
Aslak~Bakke Buan and Bethany~R. Marsh, \emph{A category of wide subcategories},
  Int. Math. Res. Not. IMRN \textbf{rnz082} (2019).

\bibitem[BR87]{butler_auslander}
M.~C.~R. Butler and C.~M. Ringel, \emph{Auslander-reiten sequences with few
  middle terms and applications to string algebras}, Comm. Algebra \textbf{15}
  (1987), no.~1-2, 145--179.

\bibitem[BS19]{baur_geometric}
Karin Baur and Raquel~Coelho Sim{\~o}es, \emph{A geometric model for the module
  category of a gentle algebra}, Int. Math. Res. Not. IMRN \textbf{rnz150}
  (2019).

\bibitem[BY13]{brustle_ordered}
Thomas Br\"ustle and Dong Yang, \emph{Ordered exchange graphs}, Advances in
  Representation Theory of Algebras (David~J. Benson, Hennig Krause, and
  Andrzej Skowro\'nski, eds.), EMS Series of Congress Reports, vol.~9, European
  Mathematical Society, 2013, arXiv:1302.6045.

\bibitem[CB89]{crawley_maps}
W.~W. Crawley-Boevey, \emph{Maps between representations of zero-relation
  algebras}, J. Algebra \textbf{126} (1989), 259--263.

\bibitem[{\c C}PS20]{canakci_extensions}
\.Ilke {\c C}anak{\c c}i, David Pauksztello, and Sibylle Schroll, \emph{On
  extensions for gentle algebras}, Canad. J. Math. (2020).

\bibitem[DIJ19]{demonet_tilting}
Laurent Demonet, Osamu Iyama, and Gustavo Jasso, \emph{$\tau$-tilting finite
  algebras, bricks, and $g$-vectors}, Int. Math. Res. Not. IMRN \textbf{2019}
  (2019), no.~3, 852--892.

\bibitem[DIR{\etalchar{+}}]{demonet_lattice}
Laurent Demonet, Osamu Iyama, Nathan Reading, Idun Reiten, and Hugh Thomas,
  \emph{Lattice theory of torsion classes}, arXiv:1711.01785.

\bibitem[Dug14]{dugas_tilting}
Alex Dugas, \emph{Tilting mutation of weakly symmetric algebras and stable
  equivalence}, Algebr. Represent. Theory \textbf{17} (2014), 863--884.

\bibitem[GM20]{garver_oriented}
Alexander Garver and Thomas McConville, \emph{Oriented flip graphs, noncrossing
  tree partitions, and representation theory of tiling algebras}, Glasg. Math.
  J. \textbf{62} (2020), no.~1, 147--182.

\bibitem[GMM20]{garver_categorification}
Alexander Garver, Thomas McConville, and Kaveh Mousavand, \emph{A
  categorification of biclosed sets of strings}, J. Algebra \textbf{546}
  (2020), no.~15, 390--431.

\bibitem[GP68]{gelfand_indecomposable}
I.~M. Gel'fand and V.~A. Ponomarev, \emph{Indecomposable representations of the
  {Lorentz} group}, Uspekhi Mat. Nauk \textbf{28} (1968), no.~2, 1--60.

\bibitem[HI]{hanson_tau}
Eric~J. Hanson and Kiyoshi Igusa, \emph{$\tau$-cluster morphism categories and
  picture groups}, arXiv:1809.08989.

\bibitem[Igu]{igusa_category}
Kiyoshi Igusa, \emph{The category of noncrossing partitions}, arXiv:1411.0196.

\bibitem[IRTT15]{iyama_lattice}
Osamu Iyama, Idun Reiten, Hugh Thomas, and Gordana Todorov, \emph{Lattice
  structure of torsion classes for path algebras}, B. Lond. Math. Soc.
  \textbf{47} (2015), no.~4, 639--650.

\bibitem[IT]{igusa_signed}
Kiyoshi Igusa and Gordana Todorov, \emph{Signed exceptional sequences and the
  cluster morphism category}, arXiv:1706.02041.

\bibitem[ITW]{igusa_picture}
Kiyoshi Igusa, Gordana Todorov, and Jerzy Weyman, \emph{Picture groups of
  finite type and cohomology in type ${A}_n$}, arXiv:1609.02636.

\bibitem[Jas15]{jasso_reduction}
Gustavo Jasso, \emph{Reduction of $\tau$-tilting modules and torsion pairs},
  Int. Math. Res. Not. IMRN \textbf{2015} (2015), no.~16, 7190--7237.

\bibitem[KY14]{koenig_silting}
Steffen Koenig and Dong Yang, \emph{Silting objects, simple-minded collections,
  $t$-structures and co-$t$-structures for finite-dimensional algebras},
  Documenta Math. \textbf{19} (2014), 403--438.

\bibitem[Lod00]{loday_homotopical}
Jean-Louis Loday, \emph{Homotopical syzygies}, Contemp. Math. \textbf{265}
  (2000), 99--127.

\bibitem[LP20]{lekili_derived}
Yanki Lekili and Alexander Polishchuk, \emph{Derived equivalences of gentle
  algebras via {F}ukaya categories}, Math. Ann. \textbf{376} (2020), 187--225.

\bibitem[M{\v S}17]{marks_torsion}
Frederik Marks and Jan {\v S}{\v t}ov\'i{\v c}ek, \emph{Torsion classes, wide
  subcategories, and localisations}, Bull. Lond. Math. Soc. \textbf{49} (2017),
  no.~3.

\bibitem[MVdlP83]{martinez_universal}
R.~Mart\'inez-Villa and J.A. de~la {Pe\~na}, \emph{The universal cover of a
  quiver with relations}, J. Pure Appl. Algebra \textbf{30} (1983), 277--292.

\bibitem[OPS]{opper_geometric}
Sebastian Opper, Pierre-Guy Plamondon, and Sibylle Schroll, \emph{A geometric
  model for the derived category of gentle algebras}, arXiv:1801.09659.

\bibitem[Pla19]{plamondon_tilting}
Pierre-Guy Plamondon, \emph{$\tau$-tilting finite gentle algebras are
  representation finite}, Pacific J. Math. \textbf{302} (2019), no.~2,
  709--716.

\bibitem[PPP19]{palu_kissing}
Yann Palu, Vincent Pilaud, and Pierre-Guy Plamondon, \emph{Non-kissing
  complexes and $\tau$-tilting for gentle algebras}, J. Comb. Algebra
  \textbf{3} (2019), no.~4, 401--438.

\bibitem[Rin90]{ringel_hall}
Claus~Michael Ringel, \emph{Hall algebras}, Topics in Algebra, Banach Center
  Publications \textbf{26} (1990), no.~1.

\bibitem[Sch99]{schroer_modules}
Jan Schr\"oer, \emph{Modules without self-extensions over gentle algebras}, J.
  Algebra \textbf{216} (1999), no.~1, 178--189.

\bibitem[Sch15]{schroll_trivial}
Sibylle Schroll, \emph{Trivial extensions of gentle algebras and {B}rauer graph
  algebras}, J. Algebra \textbf{444} (2015), 183--200.

\end{thebibliography}

\end{document}